\documentclass[a4paper]{amsart}
\pdfoutput=1 
\usepackage[utf8]{inputenc}
\usepackage[T1]{fontenc}
\usepackage{lmodern}
\usepackage{amssymb}
\usepackage{amsmath}
\usepackage{mathtools}
\usepackage[lite]{amsrefs}
\usepackage{verbatim}

\usepackage{enumitem}
\setlist[enumerate]{font=\textup}

\usepackage{microtype}
\usepackage[pdftitle={Topologically free minimal actions without dynamical comparison},
pdfauthor={Paolo Boldrini, Akshara Prasad},
pdfsubject={Mathematics},hidelinks]{hyperref}

\newtheorem{introthm}{Theorem}

\newtheorem{introqst}[introthm]{Question}
\newtheorem{introthrm}[introthm]{Theorem}

\newtheorem{thm}{Theorem}[section]
\newtheorem{lem}[thm]{Lemma}
\newtheorem{prop}[thm]{Proposition}
\newtheorem{cor}[thm]{Corollary}

\theoremstyle{definition}
\newtheorem{defn}[thm]{Definition}

\newtheorem*{nota}{Notation}

\theoremstyle{remark}
\newtheorem{ex}[thm]{Example}
\newtheorem{rmk}[thm]{Remark}

\newcommand*{\Gr}{\mathcal G}

\newcommand{\bb}[1]{\mathbb{#1}}
\newcommand{\defeq}{\vcentcolon=}

\def\acts{\curvearrowright}
\def\inv{^{-1}}

\newcommand{\N}{\mathbb N}
\newcommand{\cs}{\mathrm{C}^*}

\DeclareMathOperator{\Typ}{Typ}
\DeclareMathOperator{\Aut}{Aut}
\DeclareMathOperator{\Inv}{Inv}
\DeclareMathOperator{\St}{St}
\DeclareMathOperator{\At}{At}
\DeclareMathOperator{\Clop}{Clop}

\newcommand{\mon}{\text {mon}}
\newcommand{\B}{\mathcal B}
\newcommand{\OO}{\mathcal O}
\newcommand{\D}{\mathcal D}
\newcommand{\andSep}{\quad\text{and}\quad}

\newcommand{\Fraisse}{Fraïssé }
\newcommand*{\Cont}{\mathrm C}
\newcommand*{\Contc}{\Cont_\mathrm c} 
\newcommand*{\Finf}{\mathbb F_\infty}
\newcommand*{\NF}{\bigoplus_\N \mathbb F_\infty}
\newcommand*{\Cst}{\mathrm C^*}
\newcommand*{\red}{\mathrm r}
\newcommand*{\id}{\mathrm {id}}
\newcommand*{\Z}{\mathcal Z}
\newcommand{\cl}[1]{\overline{#1}}
\newcommand*{\Gu}{\Gr^{(0)}}
\newcommand*{\AutB}{\Aut(B,\mu)}
\DeclarePairedDelimiter{\norm}{\lVert}{\rVert}

\title[Topologically free minimal actions without dynamical comparison]{Topologically free minimal actions without dynamical comparison}
\author{Paolo Boldrini}
\email{paolob@chalmers.se}
\address{ Department of Mathematical Sciences, Chalmers University of
	Technology and University of Gothenburg, Göteborg SE-412 96, Sweden
}

\author{Akshara Prasad}
\email{a.prasad@uni-goettingen.de}
\address{Mathematisches Institut\\
	Georg-August-Universit\"at G\"ottingen\\
	Bunsenstra\ss e 3--5\\
	37073 G\"ottingen\\
	Germany}

\begin{document}
	
	\begin{abstract}
		We show the existence of a topologically free minimal action of~$\mathbb F_\infty$ on the Cantor space that does not have dynamical comparison. Moreover, we show that this phenomenon can happen both in the presence and in the absence of invariant measures. We also show that strict comparison of the reduced crossed product~$\cs$-algebra does not imply dynamical comparison for minimal actions.
		Our technique involves constructing a monoid which is not almost unperforated, embedding it into a countable refinement monoid and then realising it as the type semigroup associated to a dynamical system. 
	\end{abstract}
	
	\maketitle
	
	
	\section{Introduction}
	
	Given a minimal action of a countable discrete group~$G$ on a compact metrisable zero-dimensional space~$X$, there are two natural ways to compare the size of clopen subsets of~$X$: by decomposing and rearranging through the group action, or by comparing their measure with respect to invariant probability measures. The property of \emph{dynamical comparison} ensures that the measure theoretic notion determines the dynamical one.
	
	A clopen subset~$U\subseteq X$ is \emph{dynamically subequivalent} to a clopen~$V$, written~$U\preccurlyeq V$, if there are a finite clopen partition ~$U=U_1\sqcup\dots\sqcup U_n$ and group elements~$g_1,\dots,g_n\in G$ such that the clopens~$g_i U_i$ are pairwise disjoint and contained in~$V$. If~$U$ is dynamically subequivalent to~$V$, then it is not difficult to see that~$\mu(U)\leq \mu(V)$ for every invariant probability measure~$\mu$. The action is said to have dynamical comparison if a weak form of the converse implication holds:~$U\preccurlyeq V$ whenever $U$ and $V$ are nonempty and~$\mu(U)<\mu(V)$ for every invariant probability measure~$\mu$ on~$X$.
	This notion is meaningful even in the absence of invariant probability measures. In this case, dynamical comparison prescribes that for every nonempty clopen subset~$U$, no matter how small, the whole space~$X$ is dynamically subequivalent to~$U$.
	
	While similar notions have been studied earlier~\cites{GW:dyncompZ, Buck:Dynamical_comparison}, the term `dynamical comparison' was introduced by Kerr in his seminal work on classifiability of crossed products \cite{Kerr:Dynamical_Toms-Winter} as a dynamical version of strict comparison of positive elements in a C*-algebra. 
	Given a normalised quasitrace~\(\tau\) on a C*-algebra~\(A\), the dimension function induced by~\(\tau\)
	is defined to be~\(d_\tau(x)=\lim_{n}\tau(x^{1/n})\) for~\(x\in (A\otimes\mathbb{K})_+\). 
	A unital simple C*-algebra~\(A\) is said to have strict comparison if for any nonzero elements~\(a,b\in (A\otimes\mathbb{K})_+\), there exists a sequence~\((x_n)_n\) of elements of $A$ with~\(a=\lim_nx_nbx_n^*\) whenever~\(d_\tau(a)<d_\tau(b)\) for all quasitraces~\(\tau\) on~\(A\) with~\(d_\tau(b)<\infty\).
	Loosely speaking, the dynamical analogue of traces are \(G\)-invariant Borel measures that are (outer and inner) regular on the space~\(X\), making dynamical comparison an appropriate dynamical analogue of strict comparison.
	Interest in this notion rapidly grew thanks to the work of Kerr and Szabó \cite{KS_AF-and-SBP}, who showed that crossed products resulting from free minimal actions of amenable groups can be classified by the Elliott invariant whenever the actions have dynamical comparison and the small boundary property.
	
	The definition of dynamical comparison extends naturally beyond the zero-dimensional setting to arbitrary actions on locally compact Hausdorff spaces, where open subsets replace clopens in the definition of dynamical subequivalence (see Definition \ref{def:dyncomp} and \cite{Ma:type_semigroups_comparison}). In this broader setting, the condition has been verified in considerable generality over the last decade, both for actions of amenable groups \cites{Glasner-Liu:comparison, DZ_Comparison-property, KN_el-amenable-AF, Naryshkin_Pol-growth-comparison-SBP, Nasryshkin_Group-ext-AF, NP_AF-grp-dynamical-origin} and of non-amenable groups \cite{GGKN_cp-nonamenable}. 
	
	A list of open problems related to nuclear C*-algebras has been collected in \cite{Schafhauser-Tikuisis-White:99_problems}. Problem XXXIII of \cite{Schafhauser-Tikuisis-White:99_problems} asks whether a certain nice class of actions always has comparison. The following version of the problem for topologically free actions is still open.	
	
	\begin{introqst}\label{qst.mainqst}
		Does every topologically free minimal amenable action of a countable discrete group on a compact metrisable space have dynamical comparison?
	\end{introqst}
	
	After the first version of this paper appeared,  Glasner and Liu settled the problem affirmatively for all actions by countable amenable groups on zero-dimensional compact Hausdorff spaces \cite{Glasner-Liu:comparison}.
	The only known partial counterexample to this question arises from separated graph algebras; Ara and Exel showed in \cite{Ara-Exel:Dynamical_systems} that there exist group actions on compact spaces that fail dynamical comparison. However, these actions are far from minimal, since they have a fixed point by construction.
	
	A natural tool for studying properties of a group action~\(G\acts X\) is the associated type semigroup~\(S(G\ltimes X)\). The type semigroup of a group action on a zero-dimensional compact Hausdorff space~$X$ is a commutative monoid built as a quotient of the set of compactly supported continuous functions on $X$ with values in $\N$ by a certain equidecomposability relation. 
	While the original construction of type semigroups can be attributed to Tarski, who used it to study the Banach--Tarski paradox, type semigroups have received more attention in the last decade. Versions of type semigroups for groupoids have been employed to study pure infiniteness and stable finiteness of the associated groupoid C*-algebras (see \cites{Rainone-Sims:Dichotomy, Boenicke-Li:Ideal, Ma_Comparison-pure-infiniteness, Kwasniewski-Meyer-Prasad:Type_semigroups}). 
	Such a type semigroup can be viewed as a dynamical analogue of the Cuntz semigroup or of the positive cone of the K-theory of a C*-algebra.
	
	The realisation of prescribed monoids as type semigroups was studied by Wehrung \cite[Chapter 4]{Wehrung:Monoids_Boolean}, whose work shows that a broad class of commutative monoids can be realised this way. A central tool in the construction is \emph{monoid-valued measures}. Given a commutative monoid~$M$, an~$M$-valued measure on a Boolean algebra~$B$ is a faithful map~$\mu\colon B\to M$ that is additive in the sense that~$\mu(a\lor b)=\mu(a)+\mu(b)$ for orthogonal~$a,b\in B$. Among such measures, groupoid-induced and group-induced measures (see Definition \ref{defn:group_measurable}) play a crucial role. Given an $M$-valued measure $\mu$, the set of automorphisms of a Boolean algebra that preserve $\mu$ forms a group that acts naturally on the associated Stone space.
	Given a group action, one can take the associated type semigroup to be~\(M\).	
	In this article, we show that this is a universal construction with respect to actions on zero-dimensional spaces.
	
	\begin{introthrm}\label{thmint:universal_construction}
		Let \(G\) be a discrete group acting faithfully by homeomorphisms on a zero-dimensional locally compact Hausdorff space \(X\).
		There exist a Boolean ring~\(B\), a group-induced V-measure~\(\mu\colon B\to S(G\ltimes X)\) and a subgroup~\(H\leq\AutB\)
		such that~\(H\acts\St(B)\) is conjugate to~\(G\acts X\). In particular, \(G\acts X\) has dynamical comparison if and only if so does~\(H\acts\St(B)\).
	\end{introthrm}
	
	This universal construction indicates one to look at the action of~\(\AutB\) on its Stone space in order to construct any possible counter-example to Question~\ref{qst.mainqst}.  
	Dynamical comparison of an action has been shown to be equivalent to almost unperforation of the associated type semigroup in various settings (see \cites{Ara_Bonicke_Bosa_Li:type_semigroup, Kwasniewski-Prasad-Thiel-Wu:unperforation_comparison}). Recall that a commutative monoid~$M$ is almost unperforated if, whenever~$(n+1)a\leq nb$ for some~$a,b\in M$ and~$n\geq 1$, one already has~$a\leq b$. This is a  condition that helps rule out certain pathological order-theoretic behaviours. Furthermore, the type semigroup also detects minimality of the action: a dynamical system is minimal if and only if the type semigroup is simple \cite[Lemma 2.2]{Ara_Bonicke_Bosa_Li:type_semigroup}, meaning that it has no nontrivial order ideals.
	
	Combining these results, we employ a strategy to find subgroups of~\(\AutB\), with respect to a carefully chosen~\(M\), such that their actions on the Stone space are minimal and topologically free but lack comparison nevertheless.
	
	\begin{introthrm}\label{thm:main_theorem}
		There exist topologically free minimal actions of~$\mathbb F_\infty$ on the Cantor space without dynamical comparison. Moreover, these actions can be chosen to be either Bernoulli-measure-preserving, or without any invariant probability measure.
	\end{introthrm}
	
	The proof of Theorem \ref{thm:main_theorem} proceeds in four steps. We first construct a countable, simple, refinement cone which lacks almost unperforation.	We then realise such a cone as the type semigroup of the canonical action of the (usually uncountable) Polish group~$\Aut(B,\mu)$ on the Stone space~$\St(B)$, where~$B$ is the Cantor algebra, and~$\mu$ is a suitably chosen monoid-valued measure. In the third step, we show that any countable dense subgroup of~$\Aut(B,\mu)$ acts minimally and without dynamical comparison on~$\St(B)$. Even though topological freeness is not detected at the level of the type semigroup, we employ a Baire category argument to show that
	a generic countable subgroup of~$\Aut(B,\mu)$ is dense, isomorphic to~$\mathbb F_\infty$, and acts topologically freely on~$\St(B)$. The two variants of the theorem, with and without invariant measures, correspond to two different choices of the monoid in the first step.
	
	In the study of Cuntz semigroups and type semigroups, structural information about C*-algebras has long been encoded in preordered monoids. However, to the best of our knowledge, preordered monoids have not previously been used as a tool for constructing examples or counterexamples at the level of C*-algebras themselves via Cuntz or type semigroups. In this sense, the present work represents a reversal of perspective: rather than passing from C*-algebras to preordered monoids, we proceed in the opposite direction, using preordered monoids to derive new and interesting results in the setting of dynamical systems and C*-algebras.\medbreak
	
	\noindent\textit{Acknowledgments.} The authors thank Pere Ara for giving the idea about a minimal action whose type semigroup is not almost unperforated. The authors are also grateful to Eusebio Gardella, Ralf Meyer, and Hannes Thiel for several useful discussions; to Jamie Bell for pointing out Remark~\ref{rmk:not_Z_stable}; and to Bartosz Kwa\'sniewski for pointing out Remark~\ref{rmk:plain_paradoxes} to us.
	PB was supported by the Kungl. Vetenskapsakademien (MA2025-0041). AP was funded by the Deutsche Forschungsgemeinschaft (German Research Foundation) – 398436923 (RTG2491).\medbreak
	
	\noindent\textit{AI statement.} LLMs were used in the preparation of this manuscript to correct typos and for language improvements. The mathematical content was entirely produced by the authors.
	
	
	\section{Preliminaries}\label{sec.prelim}
	
	In the construction of dynamical systems lacking comparison, we will employ several algebraic structures. We use this section to furnish various definitions and constructions from monoid theory that will be relevant to future sections. A more exhaustive treatment of inverse semigroups, preordered monoids, and lattices can be found in \cites{Wehrung:Monoids_Boolean, Lawson:Inverse_semigroups_2026, Lawson:Inverse_semigroups}. We will also describe Stone duality along with its noncommutative version. The section ends by giving a brief description of dynamical comparison and recalling systems which are known to have the property. We include the proofs of some known results when they are not easily found in the literature. 
	
	\subsection{Stone duality}
	
	We call an algebraic structure \((R,+,\cdot)\) a \emph{ring} if \((R,+)\) is an abelian group, \((R,\cdot)\) is a semigroup and if \(\cdot\) is left- and right-distributive over \(+\). A ring with a multiplicative unit is called a \emph{unital ring}. A ring is said to be \emph{Boolean} if~$a^2=a$ for every element~$a$. 
	It follows from the definition that every Boolean ring \(B\) is commutative and satisfies \(2a=0\) for every \(a\in B\). 
	The natural ordering on a Boolean ring~\(B\) is given for~\(a,b\in B\) as~\(a\leq b\) if and only if~\(ab=a\).
	We define a new partial binary operation on~\(B\) as follows: whenever~\(a,b\in B\) with~\(ab=0\), set~\(a\oplus b=a+b\).
	
	A \emph{lattice} is a set~\(L\) together with binary operations~\(\lor\) and~\(\land\) which are associative, commutative, idempotent and satisfy~\(a\lor(a\land b)=a\) and~\(a\land(a\lor b)=a\) for every~\(a,b\in L\). The lattice is said to be \emph{distributive} if the operation~\(\land\) distributes over~\(\lor\). For elements~\(a,b\in L\), the natural partial order is given by~\(a\leq b\) if and only if~\(a\land b=a\) if and only if~\(a\lor b=b\).
	If \((B,+,\cdot)\) is a Boolean ring, the operations~$a\land b \defeq ab$,~$a\lor b\defeq a+b-ab$, and~$a\setminus b\defeq a-ab$ give~$B$ the structure of a \emph{generalised Boolean algebra}, that is, a distributive lattice~\((L,\lor,\land)\) with a minimal element~$0$ together with a binary operation~$\setminus$ satisfying~$0=a\land(b\setminus a)$ and~$a=(a\land b)\lor(a\setminus b)$ for all~$a,b\in L$. If the ring has a unit~$1$, then the generalised Boolean algebra~$B$ has a maximum, and~$B$ is said to be a \emph{Boolean algebra}.
	We denote by~\(\neg b\) the \emph{orthogonal complement}~\(1\setminus b\) of an element~\(b\in B\).
	
	\begin{ex}\label{ex:power_set_algebra}
		Given a set~\(X\), the power set~\(2^X\) is a Boolean algebra. Indeed, define the join to be the union of sets and the meet to be their intersection. Relative complement of sets satisfies the axioms of the operation~\(\setminus\). 
		Since~\(\emptyset\cap A=\emptyset\) and~\(X\cap A=A\) for any set~\(A\subseteq X\), the empty set is the 0-element and~\(X\) is the maximum. The orthogonal complement of a set will be its complement within~\(X\). All finite Boolean algebras are isomorphic to power sets of finite sets. However, there are infinite Boolean algebras which are not isomorphic to power sets.
	\end{ex}
	
	It is also possible to define a Boolean ring structure from the Boolean algebra operations by setting~$a\cdot b\defeq a\land b$ and~$a+b\defeq(a\lor b)\setminus (a\land b)$;~$0$ becomes the minimum and~$1$ becomes the maximum of the lattice, if exists. This interdefinability makes the study of Boolean rings and generalised Boolean algebras essentially equivalent. Throughout the paper, we will stick to the (generalised) Boolean algebra terminology, but we will use the Boolean ring symbols when convenient. 
	
	Stone established the aforementioned equivalence and further showed that such structures are in correspondence with compact zero-dimensional Hausdorff topological spaces \cites{Stone:Boolean_rings_algebras, Stone:Algebraic_Stone_duality}.
	While Stone's results were algebraic in nature, Doctor proved the following categorical equivalence in \cite{Doctor:Stone_duality}:
	
	\begin{thm}[Stone Duality]
		There is a contravariant equivalence between the category of Boolean algebras (with Boolean algebra homomorphisms) and the category of compact Hausdorff zero-dimensional spaces (with continuous maps). The equivalence assigns to a Boolean algebra its space of ultrafilters, and to a Stone space the Boolean algebra of its clopen subsets.
	\end{thm}
	
	The non-unital version of Stone duality can also be found in \cite{Doctor:Stone_duality}. We will freely move back and forth between the equivalent categories described in the previous theorem. We will denote the Stone space associated to a Boolean algebra by~\(\St(B)\). The topology on~\(\St(B)\) is generated by basic open sets of the form~\(\Omega_a=\{p\in\St(B)\colon a\in p\}\) for~\(a\in B\).
	
	\begin{rmk}\label{rmk:Polish_topology}
		Denote by~\(\Aut(\cdot)\) the group of automorphisms (usually of Boolean rings or of zero-dimensional locally compact Hausdorff spaces).	
		Given a Boolean ring~\(B\), the product topology on~\(\Aut(B)\) is carried over by Stone duality to the compact-open topology on~\(\Aut(\St(B))\).  A basic neighbourhood of~$g\in\Aut(B)$ in the product topology is 
		\[
		U(g; a_1,\dots, a_n)\defeq\{f\in\Aut(B)\colon f(a_i)=g(a_i)\quad\text{for } i=1,\dots, n\},
		\]
		where~$a_i\in B$. Given a zero-dimensional locally compact Hausdorff space~\(X\), a basic neighbourhood of~\(a\in\Aut(X)\) in the compact-open topology is~\(\{h\in\Aut(X)\colon h(U)=a(U)\}\) for compact open~\(U\subseteq X\). The spaces \(\Aut(B)\) with the product topology and \(\Aut(X)\) with the compact-open topology are Polish groups when~\(B\) is countable and~\(X\) is second countable, respectively. In this case, every subgroup is separable.
	\end{rmk}
	
	An \emph{atom} in a Boolean algebra is a nonzero element~\(a\) such that for any other element~\(b\) in the algebra with~\(b\leq a\), we have either~\(b=a\) or~\(b=0\). 
	Given a Boolean algebra~\(B\), we denote its set of atoms by~\(\At(B)\).
	Finite Boolean algebras are determined uniquely (up to isomorphism) by their set of atoms.
	In every power set algebra, singletons are atoms. Hence, by Example \ref{ex:power_set_algebra}, every atomless algebra with more than one element must be infinite.	
	Tarski proved that there is only one atomless countably infinite Boolean algebra up to isomorphism (see \cite[Chapter 16, Theorem 10]{Givant-Halmos:Boolean_algebras}). This corresponds to the Cantor space under Stone duality since the Cantor space is (up to homeomorphism) the unique nonempty, Hausdorff, perfect, compact, zero-dimensional, metrisable space. So we will call an atomless, countably infinite Boolean algebra the \emph{Cantor algebra}.
	
	A noncommutative version of Stone duality was proved by Lawson \cite{Lawson:General_Noncommutative_Stone} and it involves Boolean inverse semigroups and ample groupoids. While Boolean inverse semigroups can be considered as the noncommutative analogues of Boolean algebras, ample groupoids correspond to a generalisation of zero-dimensional spaces. Given an ample groupoid, we will later construct type semigroups associated to the Boolean inverse semigroup of compact open bisections of the groupoid. We shall now describe these terms.
	
	By an \emph{inverse semigroup} we mean a set~\(S\) endowed with an associative binary operation such that for every~\(x\in S\), there exists a unique~\(x\inv\in S\) with
	\[x=xx\inv x\quad\text{and}\quad x\inv=x\inv xx\inv.\]
	Two elements~\(x\) and~\(y\) are said to be \emph{compatible} if~\(x\inv y\) and~\(xy\inv\) are both idempotents.
	Idempotents are always compatible with each other. 
	The natural partial order on an inverse semigroup~\(S\) is given by~\(a\leq b\) if~\(a=ba\inv a\). The join of two elements in an inverse semigroup, if it exists, is their supremum with respect to the natural partial order. The meet is defined analogously to be the infimum. 
	
	\begin{defn}
		An inverse semigroup~\(S\) is said to be \emph{Boolean} if it satisfies the following conditions.
		\begin{enumerate}
			\item The idempotent lattice~\(E(S)\) is a generalised Boolean algebra.
			\item The join of any two compatible elements exists in~\(S\).
			\item For compatible elements~\(x,y\in S\) and any~\(z\in S\), both~\(zx\lor zy\) and~\(xz\lor yz\) exist and satisfy~\(z(x\lor y)=zx\lor zy\) and~\((x\lor y)z=xz\lor yz\).
		\end{enumerate}
	\end{defn}
	
	In a Boolean inverse semigroup, the Boolean algebra of idempotents $E(S)$ becomes a partial monoid when equipped with the orthogonal sum.
	
	A \emph{groupoid} is a set~\(\Gr\) together with an involutive inverse map, and a transitive and associative partial product map
	such that the product of an element and its inverse always exists, and~\(x\inv xy=y\) and~\(xyy\inv=x\) for any pair~\(x,y\) whose product exists. This is a generalisation of a group wherein there are possibly several units and multiplication is only defined partially. The range map $r$ on a groupoid sends an element~\(x\) to the unit~\(xx\inv\) while the source map $s$ sends it to~\(x\inv x\). We consider groupoids equipped with a topology which makes the product and inverse maps continuous. 
	All the groupoids that appear in this article are assumed to have a locally compact Hausdorff unit space~\(\Gu\).
	A groupoid is said to be \emph{\'etale} if the range map is a local homeomorphism. Every \'etale groupoid has a basis consisting of open bisections, that is, open subsets over which the range and source maps are homeomorphisms. An \'etale groupoid is said to be \emph{ample} if it has a zero-dimensional unit space.
	
	\begin{ex}
		Let $G\acts X$ be an action of a countable discrete group on a locally compact Hausdorff space. The \emph{transformation groupoid} $G\ltimes X$ associated to the action is the set $G\times X$ with product $(g,hx)(h,x)=(gh,x)$ and inverse $(g,x)\inv=(g\inv,gx)$. Equipped with the product topology $G\ltimes X$ is a locally compact Hausorff \'etale groupoid. If the unit space $X$ is zero-dimensional, the transformation groupoid is ample.
	\end{ex}

	\begin{ex}\label{ex:B_is_Boolean}
		Let~\(\Gr\) be a locally compact Hausdorff \'etale ample groupoid and~\(\Gu\) its unit space. For example, the transformation groupoid arising from the action of a discrete group on the Cantor space.
		The collection~\(\B\) of all compact open bisections of~\(\Gr\) forms a Boolean inverse semigroup, where the product of two bisections~\(U\) and~\(V\) is given by ~\(UV\defeq\{\alpha\beta\colon \alpha\in U, \beta\in V, s(\alpha)=r(\beta)\}\). The inverse of an element in~\(\B\) is the inverse of the bisection in the groupoid. 
		The idempotent lattice of~\(\B\) is the set~\(\OO\defeq\{U\subseteq \Gu\colon  U~\text{is compact open}\}\) and it inherits a generalised Boolean algebra structure as a sublattice of~\(2^{\Gu}\) (see Example \ref{ex:power_set_algebra}). If~\(\Gu\) is compact, then~\(\OO\) is a Boolean algebra.
		The natural order is the containment of sets. So the join of two elements, when it exists, is just their union. 
	\end{ex}
	
	Lawson \cite{Lawson:General_Noncommutative_Stone} defined categories of Boolean inverse semigroups and of ample groupoids equipped with specific kinds of morphisms. While we are not interested in these morphisms, it is worth noting Lawson's result which establishes a bijection between objects of these categories:
	
	\begin{thm}[Noncommutative Stone duality]
		The category of Boolean inverse semigroups is equivalent to the opposite category of ample groupoids.
	\end{thm}


	\subsection{Vaught measures}
	
	A \emph{preordered commutative monoid} is a commutative monoid~$(M,0,+)$ with a preorder~$\leq$ compatible with the addition. Every commutative monoid~\(M\) is a preordered monoid when equipped with the \emph{algebraic preorder}:~\(a\leq b\) if and only if there exists~\(c\in M\) with~\(a+c=b\). 
	When~\(\le\) is a
	partial order, we call~\(M\) an \emph{ordered \textup{(}commutative\textup{)} monoid}.
	
	\begin{defn}
		A \emph{cone} is a preordered commutative monoid equipped with the algebraic preorder and such that~$a+b=0$ implies~$a=b=0$ for any elements~\(a\) and~\(b\).
		A \emph{(conical) homomorphism} is a map between cones which is additive and such that the preimage of 0 is precisely 0. A \emph{V-homomorphism} $f\colon M\to N$ is a conical homomorphism that satisfies the Vaught property: whenever $f(m)=n_1+n_2$ there are elements $m_1,m_2\in M$ such that $m=m_1+m_2$ and $f(m_i)=n_i$.  
	\end{defn}
	
	In the literature, a cone is sometimes defined to also be equipped with a scalar multiplication by positive real numbers. We do not ask this of our cones. We set~\(\N=\{0,1,2,\dots\}\).
	
	\begin{defn}\label{defn:state}
		An element~\(x\in M\setminus\{0\}\) is \emph{paradoxical} if there exists~\(n\ge 1\) with~\((n+1)x\leq nx\).
		A \emph{state} on a cone~\(M\) is an additive and order-preserving map \(\nu\colon M\to [0,\infty]\) with~\(\nu(0)=0\) or, equivalently, \(\nu \not\equiv \infty\).
	\end{defn}
	
	A state is \emph{trivial} if it takes only the values \(0\) and~\(\infty\),
	and \emph{nontrivial} otherwise. 
	The following theorem by Tarski \cite{Tarski1938} shows that cones admit a dichotomy between paradoxical elements and existence of nontrivial finite states.
	
	\begin{thm}[Tarski's theorem]
		\label{thm:original_Tarski}
		In any cone~\(M\), an element~\(y\in M\setminus \{0\}\) is not paradoxical if and only if there is a state \(\nu\colon M\to [0,\infty]\) with \(\nu(y)=1\).
	\end{thm}
	
	A nonzero element~\(x\) in a cone~\(M\) is an \emph{atom} if whenever~\(x=a+b\) for elements~\(a,b\in M\), then either~\(a=0\) or~\(b=0\). This means that~\(x\) is minimal among all the nonzero elements with respect to the algebraic preorder. An \emph{order unit} of a cone~\(M\) is an element~\(u\) such that for all~$x\in M$ there is an~$n\in \bb{N}$ with~$x\leq nu$.
	
	\begin{defn}\label{defn:simple}
		A cone~\(M\) is \emph{simple} if every nonzero element is an order unit.
	\end{defn}
	
	The set of order units of a refinement cone is downward directed, that is, given any order units~\(u,v\in M\), there exists an order unit~\(w\in M\) with~\(w\leq u\) and~\(w\leq v\) (see~\cite[Lemma 1.5.4]{Wehrung:Monoids_Boolean}). In particular, if $M$ is a simple cone, then $M\setminus\{0\}$ is downward directed.
	\begin{defn}
		Let~$M$ be a cone. An \emph{$M$-valued measure} on a Boolean ring~$B$ is a map~$\mu\colon B\to M$ satisfying:
		\begin{enumerate}
			\item~$\mu(x\oplus y)=\mu(x)+\mu(y)$;
			\item~$\mu^{-1}\{0_M\}=\{0_B\}$.
		\end{enumerate}
		It is called a \emph{V-measure} (or \emph{Vaught measure}) if additionally it satisfies the following property:
		whenever~$\bar{a},\bar{b}\in M$ and~$c\in B$ satisfy~$\mu(c)=\bar{a}+\bar{b}$, there are elements~$a,b\in B$ such that~$c=a\oplus b$,~$\mu(a)=\bar{a}$, and~$\mu(b)=\bar{b}$.
	\end{defn}
	
	The operation~\(\oplus\) turns~\((B,\oplus,0)\) into a partial cone. A \emph{V-homomorphism} between partial cones is a homomorphism satisfying the Vaught property. So an \(M\)-valued V-measure is a V-homomorphism from a partial cone to a cone. 
	
	\begin{ex}\label{ex:Bernoulli_half_measure}
		The Cantor space can be presented as the infinite product space~$X\cong \{0,1\}^\N$. 
		The \emph{Bernoulli~$(\frac{1}{2},\frac{1}{2})$ measure} is defined as the product measure obtained by assigning probability~$\frac{1}{2}$ to both~$0$ and~$1$. This induces a V-measure on the Cantor algebra~$\Clop(X)$ with values in the cone~$(\mathbb Q_2)_+$.
	\end{ex}
	
	For an element~\(a\) in a Boolean ring~\(B\), set~\(B\downarrow a\defeq\{x\in B\colon x\leq a\}\). A partial automorphism of~\(B\) is an isomorphism of the two Boolean algebras~$B\downarrow a$ and~$B\downarrow b$ for~$a,b\in B$. Two partial automorphisms~\(f\colon B\downarrow a\to B\downarrow b\) and~\(g\colon B\downarrow c\to B\downarrow d\) can be composed as 
	\[gf\colon B\downarrow f\inv(bc)\to B\downarrow g(bc),\quad x\mapsto g(f(x)).\]
	
	\begin{defn}
		The set of all partial automorphisms of a Boolean ring~\(B\), denoted by~\(\Inv(B)\), together with composition is called the \emph{Munn semigroup} of~\(B\).
	\end{defn}
	
	Given a zero-dimensional locally compact Hausdorff space~\(X\), set~\(\Inv(X)\) to be the Boolean inverse semigroup of all homeomorphisms~\(h\colon U\to V\), for compact open sets~\(U,V\subseteq X\). Given a group~\(G\) acting on~\(X\), set~\(\Inv(X,G)\) to be the set of all partial homeomorphisms~\(h\colon U\to V\) that are piecewise in~\(G\).
	Formally, there exist a finite partition~\(U=\bigsqcup_{i=1}^n U_i\) and elements~\(g_1,\dots,g_n\in G\) such that~\(h(x)=g_ix\) for every~\(x\in U_i\).
	The set~\(\Inv(X,G)\) inherits a Boolean inverse semigroup structure as a subset of~\(\Inv(X)\). 
	
	When a group~\(G\) acts on a Boolean ring~\(B\) by automorphisms, we denote by~\(\Inv(B,G)\) the set of partial automorphisms that can be decomposed into group elements. More precisely, a partial automorphism~\(f\colon B\downarrow a\to B\downarrow b\) is in~\(\Inv(B,G)\) if there exist elements~\(a_1,\dots,a_n,b_1,\dots,b_n\in B\) and~\(g_1,\dots,g_n\in G\) such that~\(a=\oplus_i a_i\), \(b=\oplus_i b_i\), \(g_ia_i=b_i\) for every~\(i\) and
	\[f(x)=\bigoplus_{i=1}^n g_i(xa_i)\]
	for every~\(x\in B\downarrow a\). This forms a Boolean inverse subsemigroup of~\(\Inv(B)\).
	The action of~\(G\) on~\(B\) naturally induces an action on~\(\St(B)\).
	
	A neat way to think of the Munn semigroup is as the set of partial homeomorphisms on the Stone space of the Boolean ring. The next proposition proves that these semigroups are isomorphic.
	This is given in \cite[Proposition 4.4.10]{Wehrung:Monoids_Boolean} without proof.
	Similarly, it is proved in \cite[Proposition 4.4.16]{Wehrung:Monoids_Boolean} without proof that~\(\Inv(B,G)\) is isomorphic to~\(\Inv(\St(B),G)\).
	
	\begin{prop}\label{prop:iso_of_G-semigroups}
		Let~\(G\) be a discrete group acting on a Boolean ring~\(B\). Then
		\[
		\Inv(B)\cong\Inv(\St(B)) \andSep \Inv(B,G)\cong\Inv(\St(B),G)
		\]
		as Boolean inverse semigroups.
	\end{prop}
	\begin{proof}
		Let~\(f\colon B\downarrow a\to B\downarrow b\) be an element of~\(\Inv(B)\). 
		Given~\(p\in\Omega_a\), set~\(\rho(f)(p)\defeq\{x\in B\colon f\inv(x\land b)\in p\}\). This is an ultrafilter in~\(\Omega_b\) and~\(\rho(f)\colon \Omega_a\to\Omega_b\) is a partial homeomorphism of~\(\St(B)\) since~\(\rho(f)(\Omega_a)=\Omega_{f(a)}=\Omega_b\). We want to show that the map
		\[
		\rho\colon \Inv(B)\to\Inv(\St(B))\quad\text{with}\quad f\mapsto\rho(f)
		\]
		is an isomorphism of inverse semigroups.
		
		The map~\(\rho\) is a homomorphism. Indeed, consider~\(f_1\colon B\downarrow a_1\to B\downarrow b_1\) and~\(f_2\colon B\downarrow a_2\to B\downarrow b_2\).
		The partial homeomorphisms~\(\rho(f_1\circ f_2)\) and~\(\rho(f_1)\circ\rho(f_2)\) are both from~\(\Omega_{f_2\inv(b_2\land a_1)}\) to~\(\Omega_{f_1(b_2\land a_1)}\).
		Consider~\(p\) in the domain of~\(\rho(f_1\circ f_2)\) and~\(x\in B\). We have
		\begin{align*}
			x\in\rho(f_1\circ f_2)(p) &\iff (f_1\circ f_2)\inv(x\land(f_1(b_2\land a_1)))\in p \\
			&\iff f_2\inv(f_1\inv((x\land b_1)\land f_1(b_2\land a_1))) \in p\\
			&\iff f_2\inv(f_1\inv(x\land b_1)\land b_2\land a_1)\in p\\
			&\iff f_2\inv(f_1\inv(x\land b_1)\land b_2) \in p\\
			&\iff f_1\inv(x\land b_1)\in\rho(f_2)(p)\\
			&\iff x\in (\rho(f_1)\circ\rho(f_2))(p).
		\end{align*}

		We will now prove that~\(\rho\) is bijective.
		For any partial homeomorphism~\(g\colon\Omega_a\to\Omega_b\) and~\(x\in B\downarrow a\),
		the image~\(g(\Omega_x)\) is a clopen subset of~\(\Omega_b\). 
		Then there exists~\(y\in B\downarrow b\) with~\(g(\Omega_x)=\Omega_y\) by Stone duality.
		Set~\(\eta(g)(x)=y\).
		The maps
		\[\eta\colon \Inv(\St(B))\to\Inv(B)\quad\text{with}\quad g\mapsto\eta(g)\]
		and~\(\rho\) are mutual inverses of each other.
		Let~\(f\colon B\downarrow a\to B\downarrow b\), and let $x\leq a$.
		Then~\(\eta(\rho(f))(x)\) is the unique element~\(y\in B\) with~\(\rho(f)(\Omega_x)=\Omega_y\). Since it is not hard to show from the definition of $\rho$ that $\rho(f)(\Omega_x)=\Omega_{f(x)}$, we get that ~\(\eta(\rho(f))(x)=f(x)\), hence~\(\eta\circ\rho=\id\).
		Let~\(g\colon \Omega_a\to\Omega_b\) be a partial homeomorphism and~\(p\in\Omega_a\).
		We will show that~\(\rho(\eta(g))(p)=g(p)\). Given an element $y\in B$, it can be decomposed into $y_1\oplus y_2$ with $y_1\leq b$ and $y_2\land b=0$. Since both $\rho(\eta(g))(p)$ and $g(p)$ are ultrafilters containing $b$, $y$ is an element of either of them if and only if $y_1$ is. Therefore we can restrict the attention to elements $y\leq b$.
		For any~\(y\leq b\), we have~\(g\inv(\Omega_y)=\Omega_{\eta(g)\inv(y)}\).
		Then
		\begin{align*}
			y\in g(p)&\iff p\in g\inv(\Omega_y) \\
			&\iff p\in\Omega_{\eta(g)\inv(y)} \\
			&\iff \eta(g)\inv(y)\in p \\
			&\iff \eta(g)\inv(y\land b)\in p \\
			&\iff y\in\rho(\eta(g))(p)
		\end{align*}
		by the definition of~\(\rho\).	
		
		We will now show that the maps~\(\rho\colon \Inv(B)\to\Inv(\St(B))\) and~\(\eta\colon \Inv(\St(B))\to\Inv(B)\)
		described above exhibit the second isomorphism when restricted to~\(\Inv(B,G)\) and~\(\Inv(\St(B),G)\). To do that, it suffices to show that $\rho(\Inv(B,G))\subseteq \Inv(\St(B), G)$ and $\eta(\Inv(\St(B), G))\subseteq \Inv(B,G)$.
		Consider~\(f\colon B\downarrow a\to B\downarrow b\) in~\(\Inv(B,G)\) defined by the 
		partition~\(a=\bigoplus_{i=1}^n a_i\) and group elements~\(\{g_i\}_{i=1}^n\). 
		For any~\(x\in B\), one computes directly that 
		\(f\inv(x\land b)=\bigoplus_{j=1}^n(g_j\inv x\land a_j)\). 
		Let~\(p\in\Omega_{a_i}\), so that~\(a_i\in p\). 
		Since the~\(a_j\)'s are pairwise disjoint, \(a_j\notin p\) for~\(j\neq i\), 
		and since~\(g_j\inv x\land a_j\leq a_j\), the term~\(g_j\inv x\land a_j\) 
		does not belong to~\(p\) for~\(j\neq i\) either. 
		Therefore \(\bigoplus_j(g_j\inv x\land a_j)\in p\) if and only 
		if~\(g_i\inv x\land a_i\in p\), and we obtain:
		\begin{align*}
			x\in\rho(f)(p) &\iff f\inv(x\land b)\in p \\
			&\iff g_i\inv x\land a_i\in p \\
			&\iff x\in g_i(p),
		\end{align*}
		hence~\(\rho(f)(p)=g_i(p)\) for all~\(p\in\Omega_{a_i}\), as desired.
		
		On the other hand, let~\(h\colon \Omega_a\to\Omega_b\) be in \(\Inv(\St(B))\) be given by a partition~\(\Omega_a=\bigcup_{i=1}^n\Omega_{a_i}\)
		and group elements~\(\{g_i\}_{i=1}^n\). 
		By definition, we get that~\(\eta(h)\) is the unique partial automorphism in~\(\Inv(B)\)
		such that~\(h(\Omega_x)=\Omega_{\eta(h)x}\) for all~\(x\leq a\).
		For~\(x\leq a_i\), since~\(\Omega_x\subseteq\Omega_{a_i}\) and~\(h|_{\Omega_{a_i}}=g_i\),
		we have~\(h(\Omega_x)=g_i(\Omega_x)=\Omega_{g_ix}\).
		This shows that~\(\eta(h)(x)=g_i(x)\) for all $x\leq a_i$, hence $\eta(h)\in \Inv(B,G)$, as desired.
	\end{proof}
	
	If~$\mu$ is a measure on~$B$, we denote by~$\Inv(B,\mu)$ the set of all~$f\in\Inv(B)$ that preserve~\(\mu\), and by~$\Aut(B,\mu)$ the set of automorphisms that preserve~\(\mu\). 
	It can be easily verified that~$\Inv(B,\mu)$ is an inverse subsemigroup of~$\Inv(B)$ and that~\(\Aut(B,\mu)\) is a subgroup of~\(\Aut(B)\) (see \cite[Proposition 4.7.3]{Wehrung:Monoids_Boolean}).
	
	\begin{defn}\label{defn:group_measurable}
		A V-measure~\(\mu\) is said to be \emph{groupoid-induced} if for all~\(a,b\in B\),~\(\mu(a)=\mu(b)\) if and only if there exists~\(f\in\Inv(B,\mu)\) with~\(f(a)=b\). We say that~\(\mu\) is \emph{group-induced} if for all~\(a,b\in B\), one has ~\(\mu(a)=\mu(b)\) if and only if~\(a\) and~\(b\) can be decomposed into elements~\(a=\oplus_{i=1}^n a_i\) and~\(b=\oplus_{i=1}^n b_i\) and there exist~\(f_1,\dots,f_n\in\Aut(B,\mu)\) with ~\(f_i(a_i)=b_i\).
	\end{defn}
	
	As the names suggest, every group-induced V-measure is automatically groupoid-induced.
	
	
	\subsection{Dynamical comparison}\label{sec:dynamical_comparison}
	We will now define Kerr's notion of comparison for a dynamical system. As mentioned in the introduction, this is a dynamical analogue of strict comparison of positive elements in a C*-algebra. 
	
	\begin{nota}
		Throughout this article,~\(X\) denotes a locally compact Hausdorff space and~\(G\) denotes a discrete group that acts on~\(X\) by homeomorphisms.
	\end{nota}
	
	We say that a Borel measure~\(\mu\) on~\(X\) is \emph{regular} if the measure of every Borel set~\(A\subseteq X\) is the supremum of the measures of the compact subsets of~\(A\) and the infimum of the measures of open sets that contain~\(A\). Every Borel probability measure on a metrisable space is regular \cite[Theorem 17.10]{Kec95DescriptiveSetThy}.
	A Borel measure~\(\mu\) on~\(X\) is said to be \emph{\(G\)-invariant} if~\(\mu(A)=\mu(gA)\) for every~\(g\in G\) and every Borel set~\(A\subseteq X\). Denote by~\(M_G(X)\) the set of all \(G\)-invariant regular Borel probability measures on~\(X\). 
	
	\begin{defn}
		Let~\(U,V\subseteq X\) be open subsets. Write~\(U\preccurlyeq V\) if for every compact subset~\(K\subseteq U\), there exist~\(n\geq 1\), open sets~\(W_1,\dots,W_n\subseteq X\) and elements~\(g_1,\dots,g_n\in G\) such that
		\[
		K \subseteq W_1 \cup \cdots \cup W_n, \andSep
		g_1W_1 \sqcup \cdots \sqcup g_nW_n \subseteq V.
		\]
	\end{defn}
	
	Kerr introduced this as a dynamical analogue of Cuntz subequivalence of positive elements in a C*-algebra. In this spirit, when two sets~\(U\) and~\(V\) satisfy the above relation, we will say that~\(U\) is \emph{dynamically subequivalent} to~\(V\) or \emph{dynamically below}~\(V\). If this holds, then~\(\mu(U)\leq\mu(V)\) for every~\(\mu\in M_G(X)\). This leads to the following definition.
	
	\begin{defn}\label{def:dyncomp}
		A minimal system~\(G\curvearrowright X\) is said to have \emph{(dynamical) comparison} if whenever~\(U\) and~\(V\) are nonempty open sets in~\(X\) such that~\(\mu(U)<\mu(V)\) for every~\(\mu\in M_G(X)\), then~\(U\preccurlyeq V\).
	\end{defn}
	
	\begin{rmk}
		The definition of dynamical comparison simplifies when~\(X\) is \emph{zero-dimensional}. In this case, the action~$G\acts X$ has dynamical comparison if whenever~$A$ and~$B$ are compact open satisfying~$\mu(A)<\mu(B)$ for all~$\mu\in M_G(X)$, then there are a compact open partition~$A=A_1\sqcup\dots\sqcup A_n$ and elements~$g_1,\dots, g_n\in G$ such that the sets~$g_iA_i$ are pairwise disjoint and contained in~$B$ (see \cite[Proposition 3.6]{Kerr:Dynamical_Toms-Winter}).
	\end{rmk}
	
	We recall here some classes of actions that are known to satisfy dynamical comparison:
	
	\begin{itemize}
		\item Free minimal actions that are almost finite \cite{Kerr:Dynamical_Toms-Winter}.
		\item Actions of groups of subexponential growth on the Cantor space \cite{DZ_Comparison-property}.
		\item Minimal actions of finitely generated groups of polynomial growth on a compact metrisable space \cite{Naryshkin_Pol-growth-comparison-SBP}.
		\item Free actions of elementary amenable groups on finite-dimensional compact metrisable spaces \cite{KN_el-amenable-AF}.
		\item Free actions on finite-dimensional compact spaces of groups in the smallest class that contains all infinite groups of locally subexponential growth and is closed under taking countable direct limits and extensions on the right by an amenable group \cite{Nasryshkin_Group-ext-AF}.
		\item Free actions on finite-dimensional compact spaces of amenable topological full groups of Cantor minimal systems \cite{NP_AF-grp-dynamical-origin}.
		\item Amenable minimal actions of groups with paradoxical towers ---a strong form of non-amenability satisfied by free groups, acylindrically hyperbolic groups and many natural classes of non-amenable groups--- on compact metrisable spaces \cite{GGKN_cp-nonamenable}.
		\item The canonical action of an arbitrary group~$G$ on its StoneStone--Čech compactification~$\beta G$, or on its universal minimal flow~$\mu G$ \cite{Melleray:Clopen_type_semigroups}.
		\item Profinite actions of residually finite groups (see \cite[Proposition 3.11]{gar:beyondfreeness} for a proof).
		\item Any action of an amenable group on a nonempty compact Hausdorff zero-dimensional space \cite{Glasner-Liu:comparison}.
	\end{itemize}
	
	
	\section{Refinement cones and almost unperforation}\label{sec.constructionmonoids}
	In this section, we shall recall some useful properties of monoids and show the existence of perforated cones: the building blocks for the construction of our counterexamples. In this article, by `countable' we mean at most countable.
	
	\begin{defn}\label{def:almost_unperforated}
		A cone~\(M\) is \emph{cancellative} if~\(x+z=y+z\) implies that~\(x=y\) for all~\(x,y,z\in M\). It is \emph{$n$-divisible} if for all~$x\in M$ there is an element~$y\in M$ such that~$x=ny$. It is \emph{divisible} if it is~$n$-divisible for all~$n$.
		The cone is \emph{almost unperforated} if for any~\(x,y\in M\), we have~\(x\le y\) whenever there exists~\(n\ge 1\) 
		with~\((n+1)x\le n y\).
	\end{defn}
	
	A map $f\colon M\to N$ between cones $M$ and $N$ is an \emph{order embedding} if it is an injective conical homomorphism and it satisfies $f(x)\leq f(y)$ if and only if $x\leq y$ for all $x,y\in M$. In this case we say that $N$ is an \emph{extension} of $M$.
	
	\begin{rmk}
		If $M$ is a cone which is not almost unperforated and $N$ is an extension of $M$, then $N$ is also not almost unperforated. Indeed, let $f\colon M\to N$ be an order embedding and let $x$ and $y$ be elements of $M$ satisfying $(n+1)x\leq ny$ for some $n\in \N$ and $x\not\leq y$. Then $(n+1)f(x)\leq nf(y)$ and $f(x)\not\leq f(y)$.
	\end{rmk}
	
	\begin{ex}\label{ex:perforated_cone}
		Not every cone is almost unperforated.
		Consider the simple countable cancellative cone~$M_0=\{0,  {2},{3},\dots\}\subseteq \N$. The cone~$M_0$ is not almost unperforated, indeed~\(3\cdot 2\leq 2\cdot 3\) but~\(2\not\leq 3\) since~$M_0$ has no nonzero elements smaller than~\(2\).
	\end{ex}
	
	\begin{defn}\label{defn:refinement}
		A cone~\(M\) is said to have the \emph{refinement} property if for any \(a_1,a_2,b_1,b_2\in M\) with~\(a_1+a_2=b_1+b_2\), there exist~\(x_{11},x_{12},x_{21},x_{22}\in M\) with~\(a_i=x_{i1}+x_{i2}\) and~\(b_i=x_{1i}+x_{2i}\) for~\(i=1,2\).
		It has the \emph{strong refinement} property if whenever~$a_1+a_2=b_1+b_2$ with~$a_i,b_j\neq 0$, then there is a refinement~$(x_{i,j})_{i,j}$ with~$x_{i,j}\neq 0$.
	\end{defn}
	
	We call the~$2\times 2$-matrix~$(x_{i,j})_{i,j}$ a (strong) refinement of the equation~$a_1+a_2=b_1+b_2$.
	
	\begin{rmk}
		If~$M$ has the strong refinement property, then it has the refinement property. Indeed, suppose~$a_1+a_2=b_1+b_2$. If~$a_i,b_j\neq 0$, then there is a nonzero refinement. Assume that~$a_1=0$. Then~$x_{1,1}=0=x_{1,2}$,~$x_{2,1}=b_1$,~$x_{2,2}=b_2$ is a refinement, and similar refinements can be obtained when~$a_2=0$ or~$b_j=0$.
	\end{rmk}
	
	\begin{ex}
		Let $\mathbb{Q}_2$ be the set of dyadic rationals, that is, the set of rational numbers of the form $\frac{a}{2^k}$ with $a\in \mathbb Z$ and $k\in \N$. The cone~$(\mathbb{Q}_2)_+$ of nonnegative dyadic rationals has the strong refinement property. Indeed, given the equation~$a_1+a_2=b_1+b_2$ with~$a_1$ minimal among them, one possible refinement is obtained by~$x_{1,1}=\frac{a_1}{2}=x_{1,2}$,~$x_{2,1}=b_1-\frac{a_1}{2}$,~$x_{2,2}=b_2-\frac{a_1}{2}$. In general, every simple atomless refinement cone has the strong refinement property, see \cite[Theorem 3.4 and Remark 3.5]{OPR_The-corona-factorization-property-and-refinement-monoids}.
	\end{ex}
	
	\begin{ex}
		Not every refinement cone has the strong refinement property. Consider the equation~$1+1=1+1$ in~$\N$, or the equation~$(1,0)+(0,1)=(1,0)+(0,1)$ in~$\mathbb R_+^2$.
	\end{ex}
	
	The goal of this section is to construct two countable simple refinement cones,~$P$ and~$Q$, that are not almost unperforated and satisfy the following properties:
	\begin{enumerate}
		\item~$P$ is cancellative;
		\item  there is a V-homomorphism from~$P$ to the cone of nonnegative dyadic rationals~$(\mathbb{Q}_2)_+$;
		\item~$Q$ does not admit any nontrivial state.
	\end{enumerate}
	
	If~$M$ and~$N$ are cones, their product in the category of cones with conical homomorphisms is given by $\left((M\setminus\{0\})\times (N\setminus\{0\})\right)\cup \{(0,0)\}$ with the operation defined componentwise. We call this cone the \emph{conical product} of~$M$ and~$N$ and denote it by~$M\times_cN$ to distinguish it from the product in the category of monoids.
	
	\begin{prop}\label{prop.reducedprod}
		Let~$M$ and~$N$ be cones with the strong refinement property. Then~$M\times_cN$ has the strong refinement property. If both~$M$ and~$N$ are simple, then so is~$M\times_cN$. If both~$M$ and~$N$ are~$n$-divisible, then so is~$M\times_cN$.
	\end{prop}
	\begin{proof}
		Suppose~$(a_1,a_1')+(a_2,a_2')=(b_1,b_1')+(b_2,b_2')$ with all the four terms different from~$(0,0)$.
		Then~\(a_i,a_i',b_i,b_i'\), for~\(i=1,2\), are all nonzero by the definition of~$M\times_cN$.
		By the strong refinement property, we can find nonzero refinements~$(x_{i,j})$ and~$(x_{i,j}')$ of the equations $a_1+a_2=b_1+b_2$ respectively. Then~$(x_{i,j},x_{i,j}')$ is a refinement of the original equation. The statements on simplicity and divisibility are clear.
	\end{proof}
	
	\begin{rmk}
		The proposition is false if strong refinement is weakened to refinement: consider~$M=N=\N$ as a counterexample.
	\end{rmk}
	
	\begin{thm}\label{thm.constructionP}
		Let~$N$ be a countable~$2$-divisible simple refinement subcone of~$\mathbb R_+$. There exists a countable cancellative~$2$-divisible simple refinement cone~$P$ which is not almost unperforated and which admits a V-homomorphism~$\tau\colon P\to N$.
	\end{thm}
	\begin{proof}
		Consider the cone~$M_0=\{0,2,3,\dots\}\subseteq \mathbb N$ described in Example \ref{ex:perforated_cone}. It is countable, simple, cancellative, and lacks almost unperforation. By~\cite[Proposition 2.5 and Theorem 2.6]{Wehrung:Embedding_in_refinement_monoids}, there is an extension~$M$ of~$M_0$ which is a cancellative divisible simple refinement cone. Since~$M_0$ is countable, we may replace~$M$ by a countable subcone as follows. 
		Set~$M^{(0)} = M_0$. At each stage~$k \geq 0$, for every equation~$a_1 + a_2 = b_1 + b_2$ with~$a_i, b_i \in M^{(k)}$ and every pair~$(a,b) \in M^{(k)}\times M^{(k)}\setminus\{(0,0)\}$, adjoin to~$M^{(k)}$ a finite set of witnesses for the refinement of that equation, for the divisibility of~$a$, and for the inequality $a\leq nb$ that holds for some $n$, all of which exist in~$M$ by assumption. The resulting cone~$M^{(k+1)}$ is again countable. Setting~$M' = \bigcup_{k \geq 0} M^{(k)}$, we obtain a countable subcone of~$M$ containing~$M_0$ which is divisible and has the refinement property. 
		
		Replacing~$M$ by $M'$, we may assume without loss of generality that~$M$ is countable. Since ~$3\cdot 2\leq 2\cdot 3$, but~$2\not\leq 3$ in~$M$ as well, the cone~$M$ is also not almost unperforated. Since~$M$ is divisible, it is in particular atomless, therefore~$M$ has the strong refinement property. For the same reason~$N$ does have the strong refinement property. 
		
		Define~$P\defeq M\times_c N$. The cone~$P$ is countable and cancellative. Moreover by Proposition \ref{prop.reducedprod}, it is simple,~$2$-divisible, and it satisfies the (strong) refinement property.
		The failure of almost unperforation follows from the fact that given any $n\in N\setminus\{0\}$, we have~$3\cdot(2, 2n)\leq 2\cdot (3,3n)$ but~$(2,2n)\not\leq (3,3n)$. Consider the canonical projection~$\pi_N\colon P\to N$. This is a conical homomorphism and it satisfies the Vaught property. Indeed, suppose $b=\pi_N(a,b)=x+y$ with $x\neq 0$ and $y\neq 0$---as otherwise the decomposition would be trivial---, then $x=\pi_N(\frac{a}{2},x)$ and $y=\pi_N(\frac{a}{2},y)$ give the desired decomposition.
	\end{proof}
	
	By assuming~$N$ to be divisible, it is easy to get a divisible~$P$. The theorem is stated for~$2$-divisible cones since we will be particularly interested in the case~$N=(\mathbb Q_2)_+$.
	
	\begin{prop}\label{thm.constructionQ}
		There exists a countable divisible simple refinement cone~$Q$ which is not almost unperforated and admits no nontrivial states.
	\end{prop}
	\begin{proof}
		Consider the finite simple cone~$Q_0=\{0,2,3,4,5,6,\infty\}$, where addition is defined as usual when the sum is at most~$6$, and it is set to~$\infty$ otherwise. The element~$\infty$ satisfies~$\infty+\infty= \infty$, hence it is paradoxical. Moreover, $Q_0$ is not almost unperforated since~$3\cdot2\leq 2\cdot 3$, but~$2\not\leq 3$. Again by~\cite[Proposition 2.5 and Theorem 2.6]{Wehrung:Embedding_in_refinement_monoids},  there is an extension~$Q$ of~$Q_0$ which is a divisible simple refinement cone. Arguing as in the proof of Theorem \ref{thm.constructionP}, the extension~$Q$ can be chosen to be countable. The element~$\infty$ is paradoxical in~$Q$ as well. By Theorem \ref{thm:original_Tarski},~\(Q\) admits no nontrivial states normalised at $\infty$. Since $Q$ is simple, it does not admit any nontrivial state.  
	\end{proof}

	
	\section{Type semigroups} \label{sec.typesem}
	
	A useful way to detect dynamical comparison is by encoding dynamical equivalence of sets in the structure of a preordered commutative monoid, usually called a \emph{type semigroup}. 
	Given a group~\(G\) acting on a zero-dimensional space~\(X\), one can construct such a monoid as a quotient of~\(\Contc(X, \N)\),
	the set of compactly supported continuous functions~\(X\to\N\), modulo a relation that resembles dynamical equivalence of the supports of the functions.
	Almost unperforation of this monoid turns out to be equivalent to comparison of the associated dynamical system in several settings. Kerr \cite{Kerr:Dynamical_Toms-Winter} showed one direction of the equivalence for free minimal actions of countable discrete amenable groups on the Cantor space. The equivalence was proved in the general setting of second countable ample minimal groupoids in \cite[Lemma 3.5]{Ara_Bonicke_Bosa_Li:type_semigroup} (for a generalised definition of comparison for groupoids). The second countability assumption was eased to \(\sigma\)-compactness in \cite{Kwasniewski-Meyer-Prasad:Type_semigroups}. Ma \cite[Theorem B]{Ma:type_semigroups_comparison} proved that a version of dynamical comparison for non-minimal systems is equivalent to a property of the type semigroup which is weaker than almost unperforation. 
	
	Let~\(\Gr\) be an ample \'etale groupoid, for example the transformation groupoid~\(G\ltimes X\) associated to the action of~\(G\) on a zero-dimensional space~\(X\). Let~\(\B\) be the collection of compact open bisections of the groupoid and~\(\OO\) the set of idempotents in~\(\B\) (see Example \ref{ex:B_is_Boolean} for details).
	We will now define the type semigroup associated to~\(\Gr\) in the sense of \cite{Ara_Bonicke_Bosa_Li:type_semigroup}. 
	Further, we will show that such a type semigroup is isomorphic to the enveloping monoid of \(\B\) modulo Green's relation \(\D\). Such enveloping monoids are called \emph{type monoids} in \cites{Wehrung:Monoids_Boolean, Kudryavtseva:Type_monoids}. The equivalence of these two constructions of the type semigroup was used implicitly in \cite{Ara_Bonicke_Bosa_Li:type_semigroup}.
	
	\begin{defn}\label{defn:type_semigroup_Ara}
		Let~\(\Gr\) be an ample \'etale groupoid with a locally compact Hausdorff unit space~\(\Gu\) and let~\(\B\) be the collection of its compact open bisections. For~\(f,g\in\Contc(\Gu,\N)\), set~\(f\sim_\Gr g\) if there exist compact open bisections~\(W_1,\dots,W_n\) such that
		\[f=\sum_{i=1}^n 1_{s(W_i)}\andSep\sum_{i=1}^n 1_{r(W_i)}=g.\]
		We define the type semigroup associated to the groupoid to be the cone
		\[S(\Gr)\defeq\Contc(\Gu,\N)/\sim_\Gr\]
		with pointwise addition and equipped with the algebraic preorder.
	\end{defn}
	
	This article is mainly concerned with ample groupoids of the form~\(\Gr\defeq G\ltimes X\). In this case, we have~\(f\sim_\Gr g\) if and only if there exist group elements~\(a_1,\dots,a_n\in G\) and compact open sets~\(U_1,\dots,U_n\subseteq X\) such that
	\[
	f=\sum_{i=1}^n 1_{U_i}\andSep\sum_{i=1}^n1_{a_iU_i}=g.
	\]
	
	We want to show that the type semigroup can also be realised as the enveloping monoid of a quotient of~\(\B\).	
	Let~\((P,+)\) be a partial cone and let~\(\cl{P}=\{\cl{x}\colon x\in P\}\) be a disjoint copy of the underlying set of~\(P\).
	Define the \emph{enveloping monoid} $U_\mon(P)$ of~\(P\) to be the free commutative monoid generated by~\(\cl{P}\) subject to the following relations:
	\begin{enumerate}
		\item \(\cl{0}=0\) and
		\item  \(\cl{x}+\cl{y}=\cl{z}\) whenever~\(x+y\) is defined in~\(P\) and equals~\(z\).
	\end{enumerate}
	
	This definition is easily seen to be equivalent to the one given by Wehrung in \cite[Section 2.1]{Wehrung:Monoids_Boolean}.
	This monoid has the universal property that given any commutative monoid~\(N\) and any partial-monoid homomorphism~\(f\colon P\to N\), there exists a unique monoid homomorphism~\(\bar{f}\colon U_\mon(P)\to N\) such that~\(f=\bar{f}\circ\iota\) (where~\(\iota\) is the inclusion map~\(P\hookrightarrow U_\mon(P)\)) \cite[Proposition 2.1.7]{Wehrung:Monoids_Boolean}.
	
	For an element~\(x\) in an inverse semigroup~\(S\), set~\(d(x)=x\inv x\) and~\(r(x)=xx\inv\). For~\(x,y\in S\), Green's relations are given by~\(x\mathrel{\mathcal{L}}y\) if~\(d(x)=d(y)\);~\(x\mathrel{\mathcal{R}}y\) if~\(r(x)=r(y)\); and~\(x\mathrel{\mathcal{D}}y\) if~\(x\mathrel{\mathcal{L}\circ\mathcal{R}}y\) (equivalently if~\(x\mathrel{\mathcal{R}\circ\mathcal{L}}y\)).  The relation~\(\D\) is an equivalence relation on~\(S\). The quotient $S/\D=E(S)/\D$ can be equipped with the partial sum $[e]\oplus[f]=[e'\oplus f']$ whenever $[e']=[e]$, $[f']=[f]$, and $e'\perp f'$. With this partial monoid structure, the quotient map is a V-homomorphism.
	When~\(S\) is Boolean, denote by~\(\Typ(S)\) the enveloping monoid of the partial cone~\(S/\D\) and equip it with the algebraic preorder.
	
	\begin{ex}\label{ex:Type_Inv}
		Let~\(B\) be a Boolean ring admitting a measure~\(\mu\). The set~\(\Aut(B,\mu)\) of automorphisms of~\(B\) that preserve~\(\mu\) is a group, with the composition operation. This is a subgroup of~\(\Aut(B)\), therefore it acts canonically on $B$. The inverse of a partial automorphism~\(f\colon B\downarrow a\to B\downarrow b\) is the partial automorphism~\(f\inv\colon B\downarrow b\to B\downarrow a\) such that~\(f\inv f\) is the identity on~\(B\downarrow a\) and~\(f f\inv\) is the identity on~\(B\downarrow b\).
		The set~\(S\defeq\Inv(B,\Aut(B,\mu))\) is a Boolean inverse subsemigroup of~\(\Inv(B)\). For~\(f,h\in S\), we have~\(f\mathrel{\D} h\) if and only if there exists~\(g\in S\) such that~\(r(gf)=d(h)\). That is,~\(f\mathrel{\D} h\) if and only if their domains and ranges are isomorphic as Boolean algebras and the isomorphism can be exhibited by a partial automorphism that can be decomposed into automorphisms of~\(B\) that preserve the measure~\(\mu\).
		Since~\(f\mathrel{\D} f\inv\) for every partial automorphism~\(f\) and~\(\D\) is an equivalence relation,~\(S/\D\)
		is a commutative monoid with the composition operation. 
	\end{ex}
	
	\begin{ex}\label{ex:B/D_partial_monoid}
		Let~\(\B\) be the Boolean inverse semigroup of all compact open bisections of an ample groupoid~\(\Gr\) (see Example \ref{ex:B_is_Boolean}).
		Green's relation~\(\mathcal{D}\) is an equivalence relation on~\(\B\): for elements~\(U,V\in\B\), we have~\(U\mathrel{\mathcal{D}}V\) if and only if there exists~\(W\in\B\) with
		\begin{equation*}\label{eq:ample_Boolean_semigroup}
			s(U)=s(W)\quad\text{and}\quad r(W)=r(V).
		\end{equation*}
		So~\(U\mathrel{\D}s(U)\) for every~\(U\in\B\). Therefore,~\(\B/\mathcal{D}=\OO/\D\). For~\([U],[V]\in\B/\D\), set~\([U]\oplus[V]\defeq[U'\cup V']\) whenever~\(U'\cap V'=\emptyset\) for some~\(U'\in[U]\) and~\(V'\in[V]\). Then~\((\B/\D,\oplus)\) is a partial commutative monoid. 
		The identity element~\(0\) is the equivalence class of the empty set. 
		This monoid is conical since if~\([U]\oplus[V]=[\emptyset]\) for some~\(U,V\in\B\), then~\(U=V=\emptyset\).
	\end{ex}
	
	\begin{lem}\label{lem:sum_in_envelope}
		Let~\(\Gr\) be an ample \'etale groupoid with a locally compact Hausdorff unit space~\(\Gu\). Let $\B$ and~\(\OO\) be the set of compact open bisections of $\Gr$ and~\(\Gu\) respectively.
		Suppose~\(U_1,\dots,U_n,V_1,\dots,V_m\in\OO\) satisfy~$\sum_{i=1}^n 1_{U_i}=\sum_{j=1}^m 1_{V_j}$ in~$\Contc(\Gu)$. Then~$\sum_{i=1}^n[U_i]=\sum_{j=1}^m[V_j]$ in~$\Typ(\B)$. 
	\end{lem}
	\begin{proof}
		Let~$Y$ be the support of~$\sum 1_{U_i}$. Let~$\mathcal A$ be the Boolean subalgebra of~$\Clop(Y)$ generated by all the~$U_i$'s and~$V_j$'s. Since~$\mathcal A$ is finitely generated, it has finitely many atoms; we denote them by~$A_1,\dots,A_N$. Then for each~$i$
		\[
		\sum_{i=1}^n [U_i]=\sum_{l=1}^N c_l[A_l],
		\]
		where~$c_l=\#{\{i\colon A_l\subseteq U_i\}}$. Analogously:~$\sum_{j=1}^m[V_j]=\sum_{l=1}^N c_l'[A_l]$, where~$c_l'=\#\{j\colon A_l\subseteq V_j\}$. Since~$\sum_{i=1}^n 1_{U_i}=\sum_{l=1}^N c_l1_{A_l}$ and~$\sum_{j=1}^m 1_{V_j}=\sum_{l=1}^N c_l'1_{A_l}$ we have that~$c_l=c_l'$ for each~$l$. Hence~$\sum_{i=1}^n[U_i]=\sum_{j=1}^m[V_j]$ in~$\Typ(\B)$.
	\end{proof}
	
	\begin{prop}\label{prop:equivalence_of_types}
		Let~\(\Gr\) be an ample \'etale groupoid with a locally compact Hausdorff unit space~\(\Gu\). 
		Denote by~\(\B\) the Boolean inverse semigroup of all compact open bisections of~\(\Gr\). Then
		\[S(\Gr)\cong \Typ(\B)\]
		as refinement cones.
	\end{prop}
	\begin{proof}
		The map
		\[
		\tau\colon \OO\to S(\Gr)\quad\text{with}\quad U\mapsto[1_U]
		\]
		descends to a partial-monoid homomorphism
		\(
		\OO/\D\to S(\Gr)
		\)
		since~$U\mathcal{D}V$ implies~$1_U\sim_\Gr 1_V$.
		By the universal property of the enveloping monoid, this induces a conical homomorphism
		\[
		\tau'\colon \Typ(\B)\to S(\Gr).
		\]
		We show that this map is an isomorphism. Surjectivity follows from the fact that $S(\Gr)$ is generated by elements of the form $[1_U]$. For injectivity, suppose~$\sum_{i=1}^n[U_i]$ and~$\sum_{j=1}^m[V_j]$ are mapped to the same element in~$S(\Gr)$. Then by definition~$\sum_{i=1}^n 1_{U_i} \allowbreak \sim_\Gr \sum_{j=1}^m 1_{V_j}$, hence there are compact open bisections~$W_1,\dots, W_l$ such that~$\sum_{i=1}^n 1_{U_i} \allowbreak =\sum_{k=1}^l 1_{s(W_k)}$ and~$\sum_{j=1}^m 1_{V_j}=\sum_{k=1}^l 1_{r(W_k)}$. Therefore, we get
		\[
		\sum_{i=1}^n[U_i]=\sum_{k=1}^l[s(W_k)]=\sum_{k=1}^l[r(W_k)]=\sum_{j=1}^m[V_j]
		\]
		by Lemma \ref{lem:sum_in_envelope} and Example \ref{ex:B/D_partial_monoid}.
		This shows the isomorphism of the two cones.
		
		Finally, observe that~\(\Typ(\B)\) is a refinement cone since~\(\B\) is Boolean (see~\cite[Corollary 4.1.4]{Wehrung:Monoids_Boolean} and \cite[Proposition 5.4]{Kudryavtseva:Type_monoids}).
	\end{proof}
	
	\begin{rmk}\label{rmk:states=measures}
		States on a type semigroup are in bijection with regular \(\Gr\)-invariant Borel measures on the unit space~\(\Gu\), which are in turn in bijection with lower semicontinuous traces on the reduced groupoid C*-algebra (see Theorem 5.16 and Proposition 5.28 in \cite{Kwasniewski-Meyer-Prasad:Type_semigroups}). Given a state~\(\nu\), the corresponding measure~\(\mu\) is defined as~\(\mu(U)=\nu([1_U])\) for~\(U\in\OO\). States are constructed from given measures in the same way. 
		Similarly, a type semigroup is simple if and only if the ample groupoid is minimal (see Corollary~5.7 and Proposition 5.28 of \cite{Kwasniewski-Meyer-Prasad:Type_semigroups}) which happens when the reduced groupoid C*-algebra is simple.
	\end{rmk}
	
	A useful property of type semigroups is their invariance under passing to sufficiently large subsemigroups of~\(\B\).
	We will use this later to find countable subgroups os~\(\AutB\) that act minimally and without comparison. We will also use it to get a countable group acting without comparison from a given arbitrary group acting without comparison.
	
	\begin{lem}\label{lem:type_dense_subgroup}
		Let~\(X\) be a zero-dimensional locally compact Hausdorff space and consider~\(G\leq\Aut(X)\) with the compact-open topology. If~\(H\leq G\) is a dense subgroup, then~\(S(G\ltimes X)\cong S(H\ltimes X)\).
	\end{lem}
	\begin{proof}
		Consider~\(f,g\in \Contc(X,\N)\). We have
		\[
		f\sim_H g \implies f\sim_G g
		\]
		since~\(H\leq G\).
		Conversely, suppose there exist~\(g_1,\dots,g_n\in G\) and compact open sets~\(U_1,\allowbreak\dots,\allowbreak U_n\subseteq X\) with
		\[
		f=\sum_{i=1}^n 1_{U_i}\andSep g=\sum_{i=1}^n 1_{g_iU_i}.
		\]
		For each~\(i\), the set~\(V_i\defeq\{h\in G\colon h(U_i)=g_i(U_i)\}\) is an open neighbourhood of~\(g_i\) in~\(G\) (see Remark~\ref{rmk:Polish_topology}). Density of~\(H\) shows the existence of~\(h_i\in H\cap V_i\). Hence,~\(h_iU_i=g_iU_i\) which witness~\(f\sim_H g\).
	\end{proof}
	
	\begin{lem}\label{lem:perforated_subgroup}
		Let~\(G\) be a discrete group acting by homeomorphisms on a second countable zero-dimensional locally compact Hausdorff space~\(X\).
		There exists a countable subgroup~\(H\leq G\) such that the canonical homomorphism
		\[
		\Phi\colon S(H\ltimes X)\to S(G\ltimes X),\quad [f]_H\mapsto[f]_G    
		\]
		is a conical isomorphism.
	\end{lem}
	\begin{proof}
		Consider~\(f,f'\in\Contc(X,\N)\) with~\(f\sim_G f'\).
		Choose compact open subsets~\(U_1,\dots,U_n\subseteq X\) and elements~\(g_1,\dots,g_n\in G\) that witness the equivalence.
		Denote by~\(A_{f,f'}\) this finite set of group elements.
		The cone~\(\Contc(X,\N)\) is countable since~\(X\) is second countable. So the subgroup
		\[
		H\defeq\left<\bigcup_{\substack{f,f'\in\Contc(X,\N) \\ [f]_G=[f']_G}} A_{f,f'}\right>
		\]
		is countable.
		The map~\(\Phi\) is well-defined and surjective. To prove injectivity, suppose that~\(\Phi([f]_H)=\Phi([f']_H)\). Then the group elements that witness the equality are in~\(H\) by construction. So the same group elements witness~\(f\sim_H g\).
	\end{proof}


	\section{Fraïssé limits of measured Boolean algebras}\label{sec.FrLim}
	In this section, we show that the cones built in Section \ref{sec.constructionmonoids} can be realised as type semigroups of group actions. To do so, we give a new proof of the fact that any countable refinement cone~$(M,u)$, with order unit~\(u\), admits a groupoid-induced measure~$\mu\colon B\to M$ normalised at~$u$ (see \cite[Theorem 4.8.7]{Wehrung:Monoids_Boolean}). In our approach, the measure~$\mu$ is obtained as a \Fraisse limit. The flexibility provided by the \Fraisse limit structure will be useful for the genericity arguments in Section \ref{sec.minnocomp}.
	
	We begin by recalling some model theoretic language, and Fraïssé's theorem. For a more exhaustive description, see \cite[Section 7]{Hodges:Model_theory}.
	
	A \emph{signature}~$L$ is a set of relation symbols, function symbols, and constant symbols, each relation and function symbol equipped with an \emph{arity}. An \emph{$L$-structure}~$A$ consists of a nonempty set~$A$ together with an interpretation of each symbol in~$L$: each~$n$-ary relation symbol~$R$ is interpreted as a subset~$R^A \subseteq A^n$, each~$n$-ary function symbol~$f$ as a map~$f^A\colon  A^n \to A$, and each constant symbol~$c$ as an element~$c^A \in A$. An \emph{embedding} of~$L$-structures $\varphi\colon  A \hookrightarrow B$ is an injective map~$\varphi\colon  A \to B$ that preserves and reflects the interpretations of all symbols.
	
	\begin{thm}[Fraïssé's Theorem]\label{thm:Fraisse}
		Let~$L$ be a countable signature and let~$\mathcal{C}$ be a nonempty countable set of finitely generated~$L$-structures which has the following properties:
		\begin{enumerate}[label=(\alph*)]
			\item Hereditary property (HP): If~$A$ is a substructure of~$B\in \mathcal C$, then~$A\in \mathcal C$.
			\item Amalgamation property (AP): If~$\beta\colon A\hookrightarrow B$ and~$\gamma\colon A\hookrightarrow C$ are embeddings, then there are~$D\in\mathcal{C}$ and embeddings~$\beta'\colon B\hookrightarrow D$ and~$\gamma'\colon C\hookrightarrow D$ such that~$\beta'\beta=\gamma'\gamma$.
			\item Joint embedding property (JEP): for~$A$ and~$B$ in~$\mathcal{C}$, there is~$C\in\mathcal C$ such that both~$A$ and~$B$ embed in~$C$.
		\end{enumerate}
		Then there is an~$L$-structure~$D$ unique up to isomorphism, such that 
		\begin{enumerate}
			\item\label{thm:Fraisse1} $D$ is countable;
			\item\label{thm:Fraisse2}~$\mathcal{C}$ is the set of finitely generated substructures of~$D$ up to isomorphism;
			\item\label{thm:Fraisse3}~$D$ is homogeneous, that is, every isomorphism between finitely generated substructures of~$D$ extends to an automorphism of~$D$.
		\end{enumerate}
	\end{thm}
	
	Fix a countable refinement cone~$M\neq\{0\}$ and a nonzero element~$u\in M$. We want to apply Fraïssé's theorem to the following signature~$L$ and class of~$L$-structures~$\mathcal{C}$: 
	
	\begin{enumerate}
		\item $L$ is the language of Boolean algebras enriched with one unary relation symbol for each element of the monoid $M$. Formally $L=\{0,1,\land,\lor,\neg\}\cup \{\bar{m}\colon m\in M\}$. The relation~$\bar{m}$ should be thought of as the set of elements with measure~$m\in M$.
		\item~$\mathcal{C}$ is the class of finite~$(M,u)$-measured Boolean algebras; that is, structures in~$\mathcal{C}$ are pairs~$(B,\mu_B)$, where~$B$ is a finite Boolean algebra and~$\mu_B\colon B\to M$ is a measure with~\(\mu_B(1)=u\).
	\end{enumerate}
	
	Embeddings in the signature~$L$ are measure-preserving embeddings of Boolean algebras.
	
	\begin{prop}\label{prop.frlim}
		The class~$\mathcal{C}$ has countably many isomorphism classes. Moreover, it satisfies the hereditary property, the amalgamation property and the joint embedding property. In particular, it admits a \Fraisse limit~$(B,\mu)$.
	\end{prop}
	\begin{proof} Since~$M$ is countable, there are only countably many isomorphism classes. The hereditary property is trivial, and the joint embedding property will follow from amalgamation and the fact that the Boolean algebra~\(\{0,1\}\) with the only admissible normalised measure embeds into every element of~$\mathcal{C}$. It remains to show that~\(\mathcal{C}\) has the amalgamation property.
		
		Consider~\((A,\mu_A),(B,\mu_B),(C,\mu_C)\in\mathcal{C}\). Let~$\beta\colon(A,\mu_A)\hookrightarrow (B,\mu_B)$ and~$\gamma\colon(A,\mu_A) \allowbreak \hookrightarrow (C,\mu_C)$ be embeddings.
		Let~$a$ be an atom of~$A$. Then 
		\begin{equation*}
			\beta(a)=\bigoplus_{b\in \At(B\downarrow \beta(a))}b
			\andSep				
			\gamma(a)=\bigoplus_{c\in \At(C\downarrow \gamma(a))}c,
		\end{equation*}
		where~$\At(B\downarrow \beta(a))$ is the set of atoms of~$B$ that are below~$\beta(a)$ and similarly for~$\At(C\downarrow \gamma(a))$. Since both~$\beta$ and~$\gamma$ are measure-preserving,
		\begin{equation*}
			\sum_{b\in \At(B\downarrow \beta(a))} \mu_B(b)=\mu_B(\beta(a))=\mu_A(a)=\mu_C(\gamma(a))=\sum_{c\in \At(C\downarrow\gamma(a))} \mu_C(c).
		\end{equation*}
		By the refinement property, there are elements~$m^a_{b,c}\in M$ such that~$\sum_{c\in \At(C\downarrow\gamma(a))} m^a_{b,c} \allowbreak =\mu_B(b)$ and~$\sum_{b\in \At(B\downarrow \beta(a))} m^a_{b,c}=\mu_C(c)$ for all~$b\in \At(B\downarrow \beta(a))$ and~$c\in \At(C\downarrow\gamma(a))$. 
		Since~$1_B=\beta(1_A)=\beta(\bigoplus_{a\in\At(A)}a)= \bigoplus_{a\in \At(A)}\beta(a)$ 
		and~$1_C=\gamma(1_A)=\bigoplus_{a\in \At(A)}\gamma(a)$, the set of atoms of~$B$ and~$C$ can be decomposed as 
		\[
		\At(B)=\bigsqcup_{a\in \At(A)}\At(B\downarrow \beta(a)) \andSep
		\At(C)=\bigsqcup_{a\in \At(A)}\At(C\downarrow\gamma(a)).
		\]
		Consider the finite Boolean algebra~$D$ uniquely determined by its atoms
		\[
		\At(D)=\bigsqcup_{a\in \At(A)}\{(b,c)\in \At(B\downarrow \beta(a))\times \At(C\downarrow\gamma(a))\colon m^a_{b,c}\neq 0\}.
		\]
		Define a measure~$\mu_D\colon D\to M$ over the atoms of~\(D\) as~$\mu_D(b,c)=m^a_{b,c}$. For~$a\in A,~b\in\At(B\downarrow \beta(a))$ and~$c\in\At(C\downarrow\gamma(a))$, set
		\begin{equation*}
			\beta_a'(b)=\bigoplus_{c\colon m^a_{b,c}\neq 0}(b,c)
			\andSep
			\gamma'_a(c)=\bigoplus_{b\colon m^a_{b,c}\neq 0}(b,c).
		\end{equation*}
		Setting~\(\beta'(b)\defeq\beta_a'(b)\) for the unique \(a\in\At(A)\) with \(b\leq\beta(a)\), we get an embedding~$\beta'\colon B\hookrightarrow D$ of Boolean algebras. It is measure-preserving since for each atom~$a\in \At(A)$ and each atom~$b\in \At(B\downarrow \beta(a))$, we have 
		\[
		\mu_D(\beta'(b))=\sum_{c\in \At(C\downarrow \gamma(a))} m^a_{b,c}=\mu_B(b).
		\]
		Similarly, a measure-preserving embedding~$\gamma'$ can be defined by pasting the maps $\gamma_a'$. Finally, for every atom~$a\in \At(A)$ we have
		\[
		\beta'\beta(a)=\bigoplus_{b\in \At(B\downarrow \beta(a))}\beta'(b)=\bigoplus_{(b,c)\in\At(D)}(b,c)
		=\bigoplus_{c\in\At(C\downarrow\gamma(a))}\gamma'(c)=\gamma'\gamma(a).
		\]
		Therefore~$\mathcal C$ has the amalgamation property.
	\end{proof}
	
	Let~$(B,\mu)$ be the \Fraisse limit of the class of finite~$(M,u)$-measured Boolean algebras. Since~$\mathcal{C}$ is the class of finitely generated substructures of~$(B,\mu)$, and every element of~$B$ belongs to some finite subalgebra, it follows that~$\mu$ is defined on the whole of~$B$ and is a measure. The following theorem says that~$\mu$ is a groupoid-induced V-measure.
	
	\begin{thm}\label{thm.Fraisseprop}
		Let~$(B,\mu)$ be the \Fraisse limit of~$\mathcal C$. Then~$\mu$ is a groupoid-induced V-measure. Moreover, if elements~\(a_1,\dots,a_n,b_1,\dots,b_n\in B\) satisfy~\(\oplus_{i=1}^n a_i=\oplus_{i=1}^n b_i\) and~$\mu(a_i)=\mu(b_i)$, then there is a measure-preserving automorphism~$g\in \Aut(B,\mu)$ such that~$g(a_i)=b_i$ for all~$i=1,\dots, n$. In particular, if~$M$ is cancellative, then~$\mu$ is group-induced.
	\end{thm}
	\begin{proof}
		Let~$a\in B$, and~$\mu(a)=m_1+m_2\in M\setminus \{0\}$. Consider the Boolean subalgebra $A\defeq\{0, a,\neg a,1\}\subseteq B$, equipped with the measure~$\mu_A\defeq\mu|_A$. Let~$C\defeq2^{\{a_1,a_2,\neg a\}}$, and let~$\mu_C\colon C\to M$ be given by~$\mu_C(a_i)=m_i$ and~$\mu_C(\neg a)=\mu(\neg a)$.  Let~$f\colon A\hookrightarrow C$ be the measure-preserving embedding given by~$a\mapsto a_1\oplus a_2$ and~$\neg a\mapsto\neg a$. 
		Since~$(B,\mu)$ is the \Fraisse limit of~$\mathcal{C}$, there exists a measure-preserving embedding~$\iota\colon (C,\mu_C)\hookrightarrow (B,\mu)$. 
		The map~$\iota f\colon A\to \iota f(A)$ is then an isomorphism between finitely generated substructures of $(B,\mu)$. By Theorem \ref{thm:Fraisse}(\ref{thm:Fraisse3}), it extends to a measure-preserving automorphism of~$(B,\mu)$, which we denote by~$h$. Then~$a=h^{-1}(\iota(a_1))\oplus h^{-1}(\iota(a_2))$, and~$\mu(h^{-1}(\iota(a_i)))=m_i$, showing that~$\mu$ is a V-measure.
		
		In order to prove that~\(\mu\) is groupoid-induced, we show for all~$a\in B$ that~$(B\downarrow a,\mu|_{B\downarrow a})$ is the \Fraisse limit of finite~$(M,\mu(a))$-measured Boolean algebras. We start by verifying that whenever~$(C,\mu_C)$ is a finite ~$(M,\mu(a))$-measured Boolean algebra, it can be embedded in~$(B\downarrow a,\mu|_{B\downarrow a})$. Indeed, we have that~$\mu(a)=\mu_C(c_1)+\dots+\mu(c_n)$, where~$c_1,\dots,c_n$ are the atoms of~$C$, and by the Vaught property, there are~$a_1,\dots,a_n\in B\downarrow a$ such that~$\mu(a_i)=\mu(c_i)$ and~$a=a_1\oplus\dots\oplus a_n$. This shows that~$(C,\mu_C)$ embeds into~$(B\downarrow a,\mu|_{B\downarrow a})$. Moreover, homogeneity of~$(B\downarrow a,\mu|_{B\downarrow a})$ immediately follows from homogeneity of~$(B,\mu)$. If~$\mu(a)=\mu(b)$, then $(B\downarrow a, \mu|_{B\downarrow a})$ and $(B\downarrow b, \mu|_{B\downarrow b})$ are both \Fraisse limits of the class of finite 
		$(M,\mu(a))$-measured Boolean algebras. By uniqueness of the \Fraisse limit. It follows that they are isomorphic, hence~$\mu$ is groupoid-induced.
		
		If~$\{a_1,\dots,a_n\}$ and~$\{b_1,\dots,b_n\}$ satisfy~$\mu(a_i)=\mu(b_i)$ and $\oplus_{i=1}^n a_i=\oplus_{i=1}^n b_i$, then there is an isomorphism between the Boolean subalgebras~\(\langle a_1,\dots, a_n\rangle,~\langle b_1,\dots, b_n\rangle\subseteq B\)
		sending~$a_i$ to~$b_i$ (and fixing $1\setminus\oplus_{i=1}^n a_i$). By homogeneity, this extends to an automorphism of~$(B,\mu)$.
		
		Assume now that~$M$ is cancellative. Let~$\{a_1,\dots,a_n\}$ and~$\{b_1,\dots,b_n\}$ be two families of pairwise disjoint elements of $B$ satisfying~$\mu(a_i)=\mu(b_i)$ for all $i$. By cancellativity, $\mu(1\setminus\oplus_{i=1}^n a_i)=\mu(1\setminus\oplus_{i=1}^n b_i)$, hence the partitions of the unit~$\{a_1,\dots,a_n,1\setminus\oplus_{i=1}^n a_i\}$ and~$\{b_1,\dots,b_n,1\setminus\oplus_{i=1}^n b_i\}$ satisfy the condition of the previous paragraph. Therefore, there is an automorphism~$g\in \Aut(B,\mu)$ such that~$g(a_i)=b_i$ for all~$i$. This condition is stronger than~\(\mu\) being group-induced.
	\end{proof}
	
	Dobbertin showed that if~$B$ is a countable Boolean algebra and~$M$ is a refinement cone, then any two V-measures normalised at the same order unit are isomorphic~\cite[Theorem 4.6.8]{Wehrung:Monoids_Boolean}, in the sense that there is an automorphism of~$B$ that transforms one measure in the other. Therefore, when~$M$ is countable, every V-measure normalised at an order unit~$u \in M$ arises as a \Fraisse limit in the sense of Proposition \ref{prop.frlim}. 
	
	We recover two known results as a consequence of this observation and Theorem \ref{thm.Fraisseprop}. First, every V-measure defined on a countable Boolean algebra is groupoid-induced \cite[Theorem 4.6.8]{Wehrung:Monoids_Boolean}. Second, it is group-induced whenever~$M$ is cancellative \cite[Proposition 4.7.10 (c)]{Wehrung:Monoids_Boolean}. We point out that the cited proposition also shows that a groupoid-induced measure~$\mu$ normalised at~$u$ is group-induced as soon as there is an element~$u_0\in M$ and an integer~$m$ satisfying~$2u_0\leq u\leq mu_0$. The following observation is trivial, but very useful when dealing with non-cancellative monoids.
	
	\begin{rmk}\label{rmk.trivialgroupmeas}
		If~$M$ is a countable refinement cone and~$u_0$ is an order unit, then~$u=2u_0$ is an order unit that trivially satisfies the previous condition. Therefore, any groupoid-induced measure normalised at~$u$ is group-induced.
	\end{rmk}
	
	\begin{thm}\label{cor.Mistypesem}
		Let~$M$ be a countable refinement cone with an order unit~$u$ and let~\(B\) be a Boolean algebra. Let~$\mu\colon B\to M$ be a group-induced V-measure normalised at~$u$. Denote by~\(\Gr\) the transformation groupoid~\(\Aut(B,\mu)\ltimes\St(B)\).
		Then
		\[
		(M,u)\cong \big(S(\Gr),[\St(B)]\big).
		\]
	\end{thm}
	\begin{proof}
		The ~\((M,u)\)-valued measure ~\(\mu\) is group-measurable. We denote by~$S$ the Boolean inverse semigroup~$\Inv(B,\Aut(B,\mu))$ with composition operation (see Example \ref{ex:Type_Inv}).
		From the proofs of Proposition 4.8.2(2) and Proposition 4.4.20 of \cite{Wehrung:Monoids_Boolean}, we get that \((M,u)\) is isomorphic to~\((\Typ(S),[\id_B])\) as pointed cones. 
		
		Proposition \ref{prop:iso_of_G-semigroups} shows that the inverse semigroups~\(S\) and~\(\Inv(\St(B),\Aut(B,\mu))\) are isomorphic.
		Given any~\(h\in\Inv(\St(B),\Aut(B,\mu)))\) with partition~\(\Omega_a=\bigcup_{i=1}^n\Omega_{a_i}\) and group elements~\(\{g_i\}_{i=1}^n\), define a compact open bisection
		\[
		W_h=\bigcup_{i=1}^n\{(g_i,p)\colon p\in\Omega_{a_i}\}.
		\]
		The assignment~\(h\mapsto W_h\) defines an isomorphism between~\(\Inv(\St(B),\Aut(B,\mu))\) and the Boolean inverse semigroup~\(\B\) of compact open bisections of the transformation groupoid~\(\Gr\defeq\Aut(B,\mu)\ltimes\St(B)\) (see Example \ref{ex:B_is_Boolean}).
		So the cones~\(\Typ(S) \) and~\(\Typ(\B)\) are isomorphic.
		The latter is isomorphic to~\(S(\Gr)\), the type semigroup of~\(\Gr\), by Proposition \ref{prop:equivalence_of_types}.
		The unit~\([\id_B]\) of~\(\Typ(S)\) is mapped to the equivalence class of~\(\St(B)\) under this isomorphism.
	\end{proof}
	
	This dynamical interpretation of group-induced measures allows us to obtain more information about the \Fraisse limit obtained in Proposition \ref{prop.frlim}. 
	
	\begin{prop}\label{prop.finiteoratomless}
		Let~$M$ be a countable refinement cone with an order unit~$u_0$ and set ~$u=2u_0$. Let~$\mathcal{C}$ be the class of~$(M,u)$-measured finite Boolean algebras, and let~$(B,\mu)$ be the \Fraisse limit of~$\mathcal C$. Suppose~$M$ is simple. Then either~$(M,u)\cong(\N,n)$ for some $n\geq 1$ and~$B$ is the finite Boolean algebra with~$n$ atoms, or~$B$ is the Cantor algebra.
	\end{prop}
	
	\begin{proof}
		By Theorem \ref{thm.Fraisseprop},~$\mu$ is groupoid-induced and by the choice of the order unit and Remark \ref{rmk.trivialgroupmeas}, it is also group-induced. By Theorem \ref{cor.Mistypesem}, the cone~$M$ is the type semigroup of the action~$\Aut(B,\mu)\acts \St(B)$. Simplicity of the type semigroup is equivalent to minimality of the action by \cite[Lemma 2.2]{Ara_Bonicke_Bosa_Li:type_semigroup}, hence the action on~$\St(B)$ is minimal. Suppose~$B$ has an atom. Then~$\St(B)$ has an isolated point and, by minimality, the action is transitive and all points are isolated. Compactness forces~$\St(B)$ to be finite. Therefore~$(M,u)$ is the type semigroup of a transitive action on a finite set, that is,~$(M,u)\cong(\N, n)$, where~$n$ is the number of elements of~$\St(B)$, which is equal to the number of atoms of~$B$. Otherwise,~$B$ has no atoms, so it is the Cantor algebra.
	\end{proof}
	
	We will now show that the \Fraisse limit construction of Proposition~\ref{prop.frlim} is universal in the context of dynamical systems with or without dynamical comparison.
	
	\begin{lem}\label{lem:type_group-induced_measure}
		Let~\(G\) be a group that acts on a zero-dimensional locally compact Hausdorff space~\(X\) by homeomorphisms.
		Let~\(\OO\) be the Boolean ring of all compact open subsets of~\(X\).
		Then the map
		\[\mu\colon \OO\to S(G\ltimes X)\quad\text{with}\quad\mu(U)=[1_U]\]
		is a group-induced V-measure.
	\end{lem}
	\begin{proof}
		Additivity and faithfulness of~\(\mu\) are immediate. To show the Vaught property, notice that $\mu$ is the composition of the quotient map $\mathcal O\to \mathcal O/\D$ with the canonical embedding $\mathcal O/\D\hookrightarrow \Typ(\mathcal B)$ and the isomorphism of Proposition \ref{prop:equivalence_of_types}. The first and last map are clearly V-homomorphisms, while the second one is by \cite[Corollary 4.1.4 and Definition 2.1.2]{Wehrung:Monoids_Boolean}. Therefore $\mu$ is a V-measure.
		
		We now prove that~\(\mu\) is group-induced. Suppose~\(\mu(A)=\mu(B)\) for~\(A,B\in\OO\).
		That is,~\([1_A]=[1_B]\).
		Hence, there are compact open sets~\(A_1,\dots,A_m\subseteq X\) and elements~\(g_1,\dots,g_m\in G\) such that
		\[
		1_A=\sum_{j=1}^m 1_{A_j}\andSep 1_B=\sum_{j=1}^m 1_{g_jA_j}.
		\]
		This gives compact open partitions of~\(A\) and~\(B\).
		For each~\(j\), the homeomorphism~\(g_j:X\to X\) induces a Boolean ring automorphism
		\[
		\theta_j:\OO\to\OO\quad\text{with}\quad\theta_j(U)=g_j(U).
		\]
		This automorphism preserves~\(\mu\) since~\(\mu(\theta_j(U))=[1_{g_jU}]=[1_U]=\mu(U)\) for every compact open~\(U\subseteq X\).
		Thus, \(\theta_j\in\AutB\) and~\(\theta_j(A_j)=B_j\defeq g_jA_j\).
		The other direction in the definition of a group-induced V-measure is always immediate.
	\end{proof}	
	
	\begin{thm}\label{thm:universal_construction}
		Let \(G\) be a discrete group acting faithfully by homeomorphisms on a zero-dimensional locally compact Hausdorff space \(X\).
		There exist a Boolean ring~\(B\), a group-induced V-measure~\(\mu\colon B\to S(G\ltimes X)\) and a subgroup~\(H\leq\AutB\)
		such that~\(H\acts\St(B)\) is conjugate to~\(G\acts X\). In particular, \(S(H\ltimes\St(B))\cong S(G\ltimes X)\).
		If~\(X\) is compact, \(B\) can be chosen to be unital.
	\end{thm}
	\begin{proof}
		Our candidate for the Boolean ring is~\(B\defeq\OO\), the Boolean ring of all compact open subsets of~\(X\).
		The map~\(\mu\colon B\to S(G\ltimes X)\) given by~\(U\mapsto[1_U]\) is a group-induced V-measure by Lemma~\ref{lem:type_group-induced_measure}.
		Define a group homomorphism~\(\rho\colon G\to\AutB\) by~\(\rho(g)(U)=gU\). It preserves~\(\mu\) since \(\mu(\rho(g)(U))=[1_{gU}]=[1_U]=\mu(U)\).
		The kernel of~\(\rho\) coincides with the kernel of the action~\(G\acts X\), which is trivial when the action is faithful.
		Set~\(H\defeq\rho(G)\). We have~\(H\cong G/\ker(G\acts X)\).
		Let~\(\eta\colon X\to\St(B)\) be the canonical homeomorphism. We have~\(\eta(gx)=\rho(g)\eta(x)\).
		Thus the action of \(G\) modulo its kernel is conjugate to~\(H\acts\St(B)\).
		Consider the conical homomorphism
		\[
		\psi:\Contc(X,\N)\to\Contc(\St(B),\N)\quad\text{ with }\quad\psi(f)=f\circ\eta\inv.
		\]
		Suppose \(f,f'\in\Contc(X,\N)\) with~\(f\sim_G f'\). Then there exist compact open sets~\(U_1,\dots,\allowbreak U_n\subseteq X\) and~\(g_1,\dots,g_n\in G\) such that
		\[
		f=\sum_{i=1}^n1_{U_i}\andSep f'=\sum_{i=1}^n1_{g_iU_i}.
		\]
		Then,
		\[
		\psi(f)=\sum_{i=1}^n1_{\Omega_{U_i}}\andSep  \psi(f')=\sum_{i=1}^n1_{\Omega_{g_iU_i}}=\sum_{i=1}^n1_{\rho(g_i)\Omega_{U_i}}
		\]
		where \(\Omega_{U_i}=\{p\in\St(B)\colon U_i\in p\}\).
		Hence, \(\psi(f)\sim_H\psi(f')\) since~\(\rho(g_i)\in H\).
		The converse can be proved by the same argument. Therefore, \(S(H\ltimes\St(B))\cong S(G\ltimes X)\).
	\end{proof}
	
	As a corollary to the previous statement, we show that finding an amenable action on a zero-dimensional space without comparison is equivalent to finding an amenable subgroup of~\(\AutB\) which acts without comparison on the associated Stone space.
	Let~\(\text{Prob}(G)\) denote the space of probability measures on~\(G\), that is
	\[
	\text{Prob}(G)=\left\{ \xi\in\ell^1(G)\colon \xi\geq 0, \sum_{g\in G} \xi(g)=1\right\}.
	\]
	There are three topologise that~\(\text{Prob}(G)\) inherits from~\(\ell^1(G)\): norm topology, weak* topology, and the topology of pointwise convergence.
	All these topologies coincide on~\(\text{Prob}(G)\) as noted in \cite{Monod:Top_amenability}.
	That is, a map~\(m\colon X\to\text{Prob}(G)\) is continuous if and only if for each convergent net~\(x_i\to x\) in~\(X\), 
	we have~\(m(x_i)(g)\to m(x)(g)\) for all~\(g\in G\). 
	The map~\(g\cdot \xi(h)=\xi(g\inv h)\) defines the natural action of~\(G\) on~\(\ell^1(G)\).
	
	\begin{defn}
		An action of a discrete group~\(G\) on a locally compact Hausdorff space~\(X\) is said to be (\emph{topologically}) \emph{amenable} if there exists a net of continuous maps~\(m_i\colon X\to\text{Prob}(G)\) such that for every~\(g\in G\),
		\[
		\lim_{i}\left(\sup_{x\in K} \norm{gm_i(x) - m_i(gx)}_1\right) =0,
		\]
		for each compact set~\(K\subseteq X\).
	\end{defn}
	
	We first prove that the restriction of an amenable action to any subgroup is still an amenable action.
	
	\begin{lem}\label{lem:amenable_action_inheritance}
		Let~\(G\) be a discrete group acting continuously on a locally compact Hausdorff space~\(X\). If~\(G\acts X\) is topologically amenable, then so is the restricted action of every subgroup of~\(G\).
	\end{lem}
	\begin{proof}
		Let~\(H\leq G\) be a subgroup. Choose a transversal~\(T\subseteq G\) for the left cosets~\(H\setminus G\). Then for every~\(a\in G\), there exist unique~\(h\in H\) and~\(t\in T\) with~\(a=ht\).
		Define
		\begin{equation*}
			\alpha\colon G\to H\quad\text{with}\quad \alpha(ht)=h.
		\end{equation*}
		This map has the property that~\(\alpha(ab)=a\alpha(b)\) for every~\(a,b\in G\).
		Denote by~\(\widetilde{\alpha}\colon\ell^1(G)\to\ell^1(H)\) the pushforward map
		\[
		\widetilde{\alpha}(\xi)(h)=\sum_{\substack{a\in G \\ \alpha(a)=h}}\xi(a)=\sum_{t\in T}\xi(ht),
		\]
		where the last term is absolutely convergent since~\(\xi\in\ell^1(G)\).
		
		We shall first prove a few properties of~\(\widetilde{\alpha}\).
		Firstly, \(\widetilde{\alpha}\) is contractive:
		\[
		\norm{\widetilde{\alpha}(\xi)}_1=\sum_{h\in H}\left| \sum_{t\in T}\xi(ht) \right|\leq \sum_{h\in H}\sum_{t\in T}
		\left|\xi(ht)\right|=\norm{\xi}_1.
		\]
		Secondly, \(\widetilde{\alpha}(\text{Prob}(G))\subseteq\widetilde{\alpha}(\text{Prob}(H))\). Indeed, if~\(\xi\geq 0\), then~\(\widetilde{\alpha}(\xi)\geq 0\), and
		\[
		\sum_{h\in H}\widetilde{\alpha}(\xi)(h)=\sum_{h\in H}\sum_{t\in T}\xi(ht)=\sum_{a\in G}\xi(a)=1.
		\]
		Thirdly, \(\widetilde{\alpha}\) is an equivariant map with respect to the canonical actions of~\(H\) on~\(\ell^1(G)\) and~\(\ell^1(H)\). Indeed, for every~\(k\in H\),
		\begin{equation*}
			\widetilde{\alpha}(k\xi)(h)
			=\sum_{t\in T}(k\xi)(ht)=\sum_{t\in T}\xi(k\inv ht)=\widetilde{\alpha}(\xi)(k\inv h)=(k\widetilde{\alpha}(\xi))(h).
		\end{equation*}
		
		In order to show that~\(H\acts X\) is amenable, define continuous maps
		\[
		n_i\colon X\to\text{Prob}(H),\quad\text{with}\quad n_i(x)=\widetilde{\alpha}(m_i(x))
		\]
		where \(\{m_i\}_i\) witnesses amenability of~\(G\acts X\).
		Each~\(n_i\) is continuous since it is the composition of the bounded linear map~\(\widetilde{\alpha}\) with the continuous map~\(m_i\).
		Let~\(l\in H\). By equivariance of~\(\widetilde{\alpha}\), we have
		\(
		ln_i(x)=l\widetilde{\alpha}(m_i(x))=\widetilde{\alpha}(lm_i(x)).
		\)
		Using this and the fact that~\(\widetilde{\alpha}\) is contractive, we get
		\begin{align*}
			\norm{ln_i(x)-n_i(lx)}
			&=\norm{\widetilde{\alpha}(lm_i(x))-\widetilde{\alpha}(m_i(lx))}_1\\
			&=\norm{\widetilde{\alpha}(lm_i(x)-m_i(lx))}_1\\
			&\leq\norm{lm_i(x)-m_i(lx)}_1.
		\end{align*}
		Therefore, for every compact~\(K\subseteq X\), we have
		\[
		\sup_{x\in K}\norm{ln_i(x)-n_i(lx)}_1\leq\sup_{x\in K}\norm{lm_i(x)-m_i(lx)}_1\to 0
		\]
		due to amenability of the \(G\)-action and since~\(k\in H\leq G\).
		Therefore, \(H\acts X\) is topologically amenable by~\(\{n_i\}_i\).
	\end{proof}
	
	\begin{cor}\label{thm:amenability_reduction}
		The following statements are equivalent:
		\begin{enumerate}
			\item\label{thm:amenability_reduction1} There is a countable group~\(G\), a zero-dimensional second countable locally compact Hausdorff space~\(X\) and a minimal amenable action~\(G\acts X\) without dynamical comparison.
			\item\label{thm:amenability_reduction2} 
			There exist a countable refinement cone~\(M\), a countable Boolean ring~\(B\), a group-induced V-measure~\(\mu\colon\allowbreak B\to M\)
			and a subgroup~\(H\leq\AutB\) such that~\(H\acts\St(B)\) is amenable and \(S(H\ltimes\St(B))\) is simple and not almost unperforated.
		\end{enumerate}
	\end{cor}
	\begin{proof}
		Assume (\ref{thm:amenability_reduction2}). 
		The Stone space~\(\St(B)\) is second countable since~\(B\) is countable.
		By Lemma~\ref{lem:perforated_subgroup}, there exists a countable subgroup~\(G\leq H\) such that~\(S(G\ltimes\St(B))\) is simple and not almost unperforated.
		The action~\(G\acts\St(B)\) is minimal since the type semigroup is simple (see Remark~\ref{rmk:states=measures}), and amenable by Lemma~\ref{lem:amenable_action_inheritance}.
		The action
		\(G\acts\St(B)\) lacks dynamical comparison by \cite[Corollary B]{Ara_Bonicke_Bosa_Li:type_semigroup}.
		
		Assume (\ref{thm:amenability_reduction1}).
		Denote by~\(B\) the Boolean ring of compact open subsets of~\(X\).
		The map~\(\mu\colon B\to S(G\ltimes X)\) with~\(U\mapsto[1_U]\) is
		a group-induced V-measure by Lemma~\ref{lem:type_group-induced_measure}.
		By Theorem~\ref{thm:universal_construction}, there exists a subgroup~\(H\leq\AutB\) such that~\(S(H\ltimes\St(B))\cong S(G\ltimes X)\).
		This type semigroup is countable since~\(X\) is second countable; has refinement by Proposition~\ref{prop:equivalence_of_types}; is simple since~\(G\acts X\) is minimal, see Remark~\ref{rmk:states=measures}; and is not almost unperforated by \cite[Corollary B]{Ara_Bonicke_Bosa_Li:type_semigroup}.
		So~\(H\acts\St(B)\) is minimal, see Remark~\ref{rmk:states=measures}.
		In the proof of Lemma~\ref{lem:amenable_action_inheritance}, if we take~\(\alpha\) to be the quotient homomorphism \(G\to G/\ker(G\acts X)\),
		then we see that amenability of the induced faithful action~\(H\acts\St(B)\) follows from that of~\(G\acts X\).
	\end{proof}
	
	After the first version of this paper came out, Glasner--Liu \cite{Glasner-Liu:comparison} proved that all actions by countable discrete amenable groups on nonempty comapct Hausdorff zero-dimensional spaces have comparison. So the groups in Corollary~\ref{thm:amenability_reduction} cannot be amenable. Hence, the action~\(G\acts X\) in  Corollary~\ref{thm:amenability_reduction}(\ref{thm:amenability_reduction1}) is a topologically amenable action by a nonamenable group if and only if so is~\(H\acts\St(B)\).
	
	
	\section{Actions lacking dynamical comparison} \label{sec.minnocomp}
	In this section, we construct topologically free minimal actions of~$\mathbb F_\infty$ without dynamical comparison, and we study additional dynamical properties of such actions. 
	
	We showed in Section \ref{sec.constructionmonoids} the existence of countable~$2$-divisible simple refinement cones that are not almost unperforated. Moreover,  we constructed the \Fraisse limit of the class of finite Boolean algebras with normalised measures in such cones in Section \ref{sec.FrLim}, and showed that it consists of the Cantor algebra equipped with a group-induced measure.
	These structures will help us construct minimal dynamical systems which lack comparison.
	
	\begin{thm}\label{thm.minactnocomp}
		Let \(M\) be a countable simple refinement cone, $u_0\in M\setminus\{0\}$ and~$u=2u_0$. Let moreover~$(B,\mu)$ be the \Fraisse limit of finite~$(M,u)$-measured Boolean algebras. Then the canonical action of any countable dense subgroup~\(H\leq \Aut(B,\mu)\) on~\(\St(B)\) is minimal. 
		
		Moreover, if~\(M\) is not almost unperforated, then~\(\St(B)\) is the Cantor space and~\(H\acts\St(B)\) lacks dynamical comparison.
	\end{thm}
	
	\begin{proof}
		By Remark \ref{rmk.trivialgroupmeas} and Theorem \ref{cor.Mistypesem}, the type semigroup of the action~$\Aut(B,\mu)\acts \St(B)$ is~$M$, hence it is simple.
		Since~\(H\leq\AutB\) is dense, Lemma~\ref{lem:type_dense_subgroup} gives~\(S(H\ltimes\St(B))\cong M\). The action~\(H\acts\St(B)\) is minimal by Remark~\ref{rmk:states=measures}. 
		
		On applying Proposition~\ref{prop.finiteoratomless} and using the fact that~\(M\) is not almost unperforated, we get that~$\St(B)$ is the Cantor space.
		The action~\(H\acts\St(B)\) lacks dynamical comparison since its type semigroup is not almost unperforated, see \cite[Theorem A]{Ara_Bonicke_Bosa_Li:type_semigroup}.
	\end{proof}
	
	Cones satisfying the assumptions of Theorem~\ref{thm.minactnocomp} were constructed in Section~\ref{sec.constructionmonoids}, which leads us to the following result.
	
	\begin{cor}\label{cor:min_comp}
		There exist minimal Cantor actions by countable groups that lack dynamical comparison.
	\end{cor}
	
	
	\subsection{Topologically free minimal actions}
	We now show that dynamical comparison can also fail for topologically free minimal actions by employing Baire category arguments. We start by showing that a generic countable subgroup of~$\Aut(B,\mu)$ is dense.
	
	\begin{defn}
		An action~$G\acts X$ is \emph{topologically free} if for every~$g\in G\setminus\{e\}$, the set~$\mathrm{Fix}(g)\defeq\{x\in X\colon gx=x\}$ of fixed points has empty interior.
	\end{defn} 
	
	\begin{defn}
		Let~\(G\) be a Polish group and let~\(P\) be a property that a countable subgroup of~\(G\) could satisfy.
		We say that the property~$P$ is \emph{generic} among countable subgroups of~\(G\) if
		\[
		\{(g_n)_{n\in\N}\in G^\N\colon \langle g_n\colon n\in\N\rangle\leq G \text{ has property }P\}
		\]
		is a dense~$G_\delta$ subset of~$G^\N$. Equivalently, we say that a generic countable subgroup of~$G$ has the property~$P$.	
	\end{defn}
	
	\begin{lem}\label{lem.dense}
		Let~$Y$ be a Polish space. The set 
		\[
		D(Y)\defeq\{(y_n)_{n\in \N}\colon \{y_n\colon n\in\N\} \text{ is dense in }Y\}\]
		is a dense~$G_\delta$ subset of~$Y^\N$.
	\end{lem}
	\begin{proof}
		Since~\(Y\) is separable and metrisable, it admits a countable basis, say~$\{U_k\colon k\in \N\}$. 
		Let~\(\pi_n\colon Y^\N\to Y\) be the projection map of the \(n^{\text{th}}\) coordinate.
		The set 
		\[
		D_{k}\defeq \{(y_n)_n\colon \exists n_0~\text{with}~y_{n_0}\in U_k\}=\bigcup_{n\in \N}\pi_n\inv(U_k)
		\]
		is open, being the union of open sets.
		The set~\(D(Y)\) is~\(G_\delta\) since it can be written as 
		\[D(Y)=\bigcap_{k\in\N}D_k.\]
		It remains to show that each~$D_k$ is dense, since that would show density of~\(D(Y)\) by the Baire category theorem.
		To this end, consider the basic open set 
		\[
		\emptyset\neq V= W_{1}\times W_{2}\times\dots\times W_{m}\times Y\times\dots 
		\]
		and take points~$y_i\in W_{i}$. Let moreover~$y\in U_k$. Then 
		\[
		(y_1,y_2,\dots, y_m,y,y,\dots)\in V\cap D_k.
		\]
		Hence~$V\cap D_k\neq\emptyset$, showing that~$D_k$ is dense.
	\end{proof}
	
	Consider the group $\mathbb F_\infty=\langle x_1,x_2,\dots\rangle$. Given a word~$w=x_{i_1}^{\epsilon_1}x_{i_2}^{\epsilon_2}\dots x_{i_n}^{\epsilon_n}\in\mathbb{F}_\infty$, and a sequence~$(g_i)_{i\in \N}$ of elements of a topological group~$G$, we denote by~$w((g_i)_{i\in\N})$ the element~$g_{i_1}^{\epsilon_1}g_{i_2}^{\epsilon_2}\dots g_{i_n}^{\epsilon_n}\in G$. Seen this way, the word $w$ is a function from $G^\N$ to $G$ which is continuous since the group operations are continuous on $G$. 
	
	\begin{nota}
		Throughout this subsection, we denote by~\(M\) a countable $2$-divisible simple refinement cone with a fixed nonzero element~\(u\). The \Fraisse limit of the class of finite $(M,u)$-measured Boolean algebras will be denoted by~$(B,\mu)$. We set $G=\Aut(B,\mu)$ and $X=\St(B)$.
		Given a nontrivial word $w\in \mathbb F_\infty$ and a nonzero element $c\in B$, set
		\[
		F_{w,c}\defeq\{(g_n)_n\in G^\N\colon w((g_n)_n)|_{ c}\neq \id|_{c}\}
		\]
		and 
		\[
		W\defeq \{w\in \mathbb F_\infty\colon F_{w,c} \text{ is dense for all } c\in B\setminus \{0\}\}.
		\]
	\end{nota}

	The sets~\(F_{w,c}\) are open and in Proposition \ref{thm.generictopfreenessnew} we will show they are dense. That is, we will show that~\(W=\Finf\setminus\{e\}\).
	We start with some preliminary observations. Firstly,~$W$ is closed under taking inverses. Indeed if $w((g_n)_n)$ is not the identity on $c$, then $w^{-1}((g_n)_n)$ is not the identity on $c$ either.
	Secondly, every nonempty word in $\mathbb F_\infty$ can be written in reduced form $x_{n_1}^{\epsilon_1}x_{n_2}^{\epsilon_2}\dots x_{n_l}^{\epsilon_l}$, where \(\epsilon_i\in\{-1,1\}\) for every \(i\) and if two consecutive letters are the same, then so are the exponents. We say that $l$ is the length of $w$. For each word $w\in \mathbb F_\infty$ and elements $a,b\in B$, the set $U=\{(g_n)_n\in G^\N\colon w((g_n)_n)a=b\}$ is open, being the preimage of the open set $\{g\in G\colon ga=b\}\subseteq G$ under $w$.
	
	\begin{rmk}\label{rmk:preliminary}
		Suppose $g\in G$ is not the identity on a clopen $c$. Then there is $x\in c$ with $g(x)\neq x$. Let $a$ and $b$ be disjoint clopen neighbourhoods of $x$ and $g(x)$ respectively. Then $d=a\land c\land g^{-1}b$ is a nonempty neighbourhood of $x$ satisfying $gd\land d=0$. In particular, if $(g_n)_n\in F_{w,c}$, then there is $0<d\leq c$ such that $w((g_n)_n)d\land d=0$.		
	\end{rmk}
	
	\begin{lem}\label{lem.Wclosedbyconj}
		If $w\in W$, then for all $k\in\N$ the word $x_kwx_k\inv$ is in $W$.
	\end{lem}
	\begin{proof}
		Let $c\in B\setminus\{0\}$ and $V\subseteq G^\N$ be a nonempty open subset. We show that $F_{x_kwx_k\inv,c}\cap V\neq \emptyset$. Let $(h_n)_n\in V$ and consider the open set 
		\[
		V'=V\cap \{(g_n)_n\in G^\N\colon g_k\inv c=h_k\inv c\}.
		\]
		The set $V'$ is nonempty since it contains $(h_n)_n$, therefore $V'$ intersects the dense open set $F_{w,h_k\inv c}$. Let $(g_n)_n$ be an element in this intersection. This is an element of $V$ that satisfies $w((g_n)_n)|_{h_k\inv c}\neq \id|_{h_k\inv c}$. Since $(g_n)_n\in V'$, we have that $h_k\inv c=g_k\inv c$, hence $w((g_n)_n)|_{g_k\inv c}\neq \id|_{ g_k\inv c}$. This is equivalent to $\big((x_kwx_k\inv)((g_n)_n)\big)|_{ c}\neq \id|_{ c}$. Therefore $V\cap F_{x_kwx_k\inv,c}\neq \emptyset$.
	\end{proof}
	
	The following lemma is the technical core of the proof of Proposition \ref{thm.generictopfreenessnew}.
	
	\begin{lem}\label{lem:simplifytopfreeness}
		Let~\(k\in\N\),~\(w\) be a word in $\Finf$,~\(c\in B\) be a nonzero element, ~\(V\subseteq\ G^\N\) be a nonempty open subset, and~\((g_n)_n\) be an element of $V$ satisfying $g_kh|_{c}=\id|_{c}$, where $h=w((g_n)_n)$.
		Then there exist~\(g'\in G\) and nonzero disjoint elements $d_1,d_2\leq hc$ such that
		\begin{enumerate}
			\item $(g_1,\dots,g_{k-1},g',g_{k+1},\dots)\in V$;
			\item $g'|_{B\downarrow 1\setminus hc}=g_k|_{B\downarrow 1\setminus hc}$ and $(g')\inv|_{B\downarrow 1\setminus c}=g_k\inv|_{B\downarrow 1\setminus c}$;
			\item $g'(d_1)=g_k(d_2)$ and $g'(d_2)=g_k(d_1)$.
		\end{enumerate}
	\end{lem}
	\begin{proof}
		Without loss of generality, we can assume that there exist a positive integer $I\geq k$ and elements $a_1,\dots, a_m\in B$ such that $V$ has the form $V=\prod_{i\in \N}U_i$, where ~\(U_i=\{h\in G\colon h(a_j)=g_i(a_j)\text{ for }j=1,\dots,m\}\) for $i\leq I$, and $U_i=G$ for $i>I$.
		
		Consider the subalgebra 
		\[
		B_0\defeq\langle c, h(c), a_1,\dots, a_m\rangle\subseteq B
		\]
		and enumerate its atoms as~$\At(B_0)=\{b_0,b_1,\dots, b_N\}$. We can assume without loss of generality that~$b_0\leq h(c)$, and hence~$g_k(b_0)\leq g_kh(c)=c$. Since~$M$ is~$2$-divisible and~$\mu$ is a~$V$-measure, there is a partition ~$b_0=d_1\oplus d_2$ such that~$\mu(d_1)=\mu(d_2)$. Consider now the subalgebra~$B_1$ with atoms
		\[
		\At(B_1)=\{d_1,d_2,b_1,\dots, b_N\}.
		\]
		Then clearly~$B_0\subseteq B_1$. Define the partial automorphism~$g''$ with domain~$B_1$ given by~$g''(b_i)=g_k(b_i)$ for~$i=1,\dots, N$,~$g''(d_1)=g_k(d_2)$, $g''(d_2)=g_k(d_1)$. Then~$g''$ is a measure-preserving partial automorphism between finitely generated substructures of $(B,\mu)$, hence by the homogeneity property (Theorem \ref{thm:Fraisse} (\ref{thm:Fraisse3})) it can be extended to an automorphism of~\(B\), still denoted by~$g''$. Since $g_kb_0=g''b_0$, we can define a measure-preserving automorphism $g'\in G$ by setting it equal to $g''$ on $B\downarrow b_0$ and equal to $g_k$ on $B\downarrow1\setminus b_0$. Since $b_0\leq hc$, we have that $g'|_{ B\downarrow 1\setminus hc}=g_k|_{B\downarrow 1\setminus hc}$ and $(g')\inv|_{B\downarrow 1\setminus c}=g_k\inv|_{B\downarrow 1\setminus c}$. Since $g'|_{B_0}=g_k|_{B_0}$ and all the elements $a_j\in B_0$, we have that $g'\in U_k$, hence $(g_1,\dots,g_{k-1},g',g_{k+1},\dots)\in V$. Finally, the conditions $g'(d_1)=g_k(d_2)$ and $g'(d_2)=g_k(d_1)$ are true by construction.	
	\end{proof}
	
	\begin{prop}\label{thm.generictopfreenessnew}
		For any~\(w\in\Finf\setminus\{e\}\) and~\(c\in B\setminus\{0\}\), the set
		\[
		F_{w,c}=\{(g_n)_n\in \Aut(B,\mu)^\N\colon w((g_n)_n)|_c\neq\id|_c\}
		\]
		is open and dense in~\(\Aut(B,\mu)^\N\).
	\end{prop}
	\begin{proof}
		Set $G=\Aut(B,\mu)$. The fact that $F_{w,c}$ is open follows from writing it as the following union of open sets  
		\[
		F_{w,c}=\bigcup_{0<b\leq c}\{(g_n)_n\in G^\N\colon w((g_n)_n)b\land b= 0\}.
		\]
		
		We prove density by induction on the length~\(l\geq 1\) of a nontrivial word~\(w\in\Finf\).
		If $l=1$, then $w$ is either $x_k$ or $x_k\inv$ for some $k\in\N$. Since $W$ is closed under taking inverses we can assume $w=x_k$. Let $V$ be a nonempty open subset of $G^\N$, we show that $V\cap F_{x_k,c}$ is nonempty. Let $(g_n)_n\in V$, and suppose $g_k|_{ c}=\id|_{ c}$, as otherwise we would be done. Then the hypothesis of Lemma \ref{lem:simplifytopfreeness} are satisfied for the empty word. Therefore there is $g'\in G$, and nonzero disjoint elements $d_1,d_2\leq c$ such that $(g_1,\dots,g_{k-1},g',g_{k+1},\dots)\in V$, $g'(d_1)=d_2$, and $g'(d_2)=d_1$. This shows that $g'|_{c}\neq \id|_{c}$, hence $(g_1,\dots,g_{k-1},g',g_{k+1},\dots)\in V\cap F_{x_k,c}$.
		
		Suppose now $l\geq 2$ and let $w=x_{n_1}^{\epsilon_1}x_{n_2}^{\epsilon_2}\dots x_{n_l}^{\epsilon_l}$.  Using the fact that $W$ is closed under taking inverses, together with Lemma \ref{lem.Wclosedbyconj}, we can also assume $\epsilon_1=1$ and $x_{n_l}^{\epsilon_l}\neq x_{n_1}\inv$. Moreover, up to rearranging the variables we can assume $n_1=1$. This means that $w$ can be written in the following form:
		\[
		w= x_1w_1x_1^{\epsilon_1}w_2x_1^{\epsilon_2}\dots w_k x_1^{\epsilon_k}w_{k+1},
		\]
		for some~\(k\), where the words $w_j$ do not contain $x_1$ and either $\epsilon_k=1$ or $w_{k+1}$ is nonempty. In what follows, we set 
		\[
		w_{i,j}\defeq
		\begin{cases}
			w_ix_1^{\epsilon_i}\dots w_j &\text{ for }1\leq i<j\leq k+1,\\
			w_j &\text{ for }i=j
		\end{cases}
		\]
		In particular, $w=x_1w_{1,k+1}$. We show that $w\in W$ by showing that~$V\cap F_{w,c}$ is nonempty for any $c\in B\setminus\{0\}$ and nonempty open set $V\subseteq G^\N$.
		
		For $\epsilon\in\{-1,1\}$ and $j=1,\dots, k$, consider the function $e_j^\epsilon\colon G^\N\times B\to B$ defined by
		\[
		\begin{split}
			e_j^1((g_n)_n;b)&=w_{k+2-j,k+1}((g_n)_n)b\land w_{1,k+1}((g_n)_n)b,\\
			e_j\inv((g_n)_n;b)&=w_{k+2-j,k+1}((g_n)_n)b\land b.
		\end{split}
		\]
		Let $F$ be the (finite) subset of $\mathbb F_\infty$ consisting of all the words $w_{i,j}$ for $1\leq i\leq j\leq k+1$.  
		Set $V_0=V$ and $c_0=c$. We recursively construct triples $(V_j, c_j, (g_n^{(j)})_n)$ with $j=1,\dots,k$ satisfying certain conditions to eventually show that~\(\emptyset\neq V_k\cap F_{w,c_k}\subseteq V\cap F_{w,c}\). The conditions are
		\begin{enumerate}[label=(\roman*)]
			\item\label{cond.i} $V_{j+1}\subseteq V_j$ for $j=0,\dots,k-1$ and they are open sets;
			\item $0<c_{j+1}\leq c_j$ for $j=0,\dots,k-1$;
			\item $(g^{(j)}_n)_n\in V_j$ for $j=1,\dots,k$;
			\item $f((g_n)_n)c_{j}=f((g^{(j)}_n)_n)c_{j}$ for all $f\in F$, $j=1,\dots,k-1$, and all~\((g_n)_n\in V_{j+1}\);
			\item\label{cond.v} $e^{\epsilon_{k+1-j}}_j((g^{(j)}_n)_n;c_j)=0$ for $j=1,\dots,k$.
		\end{enumerate}
		
		We start by constructing $(V_1,c_1, (g^{(1)}_n)_n)$. The construction depends on the value of $\epsilon_k$. Suppose first that $\epsilon_k=1$. Let $(g_n^{(0)})_n\in V$. Consider the nonempty open set 
		\[
		V_1\defeq V\cap \{(g_n)_n\colon f((g_n)_n)c=f((g_n^{(0)})_n)c\text{ for all $f\in F$}\}\cap F_{w_{1,k}x_1,w_{k+1}((g^{(0)}_n)_n)c},
		\]
		where to claim nonemptiness we are using the fact that $w_{1,k}x_1\in W$, true by the induction hypothesis.
		
		Let $(g_n^{(1)})_n$ be an element of $V_1$. Then 
		\[
		(w_{1,k}((g_n^{(1)})_n)g^{(1)}_1)|_{w_{k+1}((g_n^{(0)})_n)c}\neq \id|_{ w_{k+1}((g_n^{(0)})_n)c}, 
		\]
		which by construction of $V_1$ is equivalent to
		\[
		(w_{1,k}((g_n^{(1)})_n)g^{(1)}_1)|_{w_{k+1}((g_n^{(1)})_n)c}\neq \id|_{w_{k+1}((g_n^{(1)}))_nc}. 
		\]
		Therefore, by Remark \ref{rmk:preliminary}, there is $0<c_1\leq c$ such that 
		\[
		w_{1,k}((g_n^{(1)})_n)g^{(1)}_1w_{k+1}((g_n^{(1)})_n)c_1\land w_{k+1}((g_n^{(1)})_n)c_1=0,
		\]
		that is, $e_1^1((g^{(1)}_n)_n;c_1)=0$.
		
		Suppose instead $\epsilon_k=-1$. 
		The word~\(w_{k+1}\) is nonempty and strictly shorter than~$w$ by the construction of~$w$. Hence~$w_{k+1}\in W$.
		Let $V_1\defeq V\cap F_{w_{k+1},c}$, and let $(g_n^{(1)})_n\in V_1$. Then
		\[
		w_{k+1}((g_n^{(1)})_n)|_{c}\neq \id|_{c} 
		\]
		and we can find as before an element $0<c_1\leq c$ such that $w_{k+1}((g_n^{(1)})_n)c_1\land c_1=0$, that is, $e_1\inv((g^{(1)}_n)_n;c_1)=0$.
		
		Suppose now that $(V_j,c_j,(g^{(j)}_n)_n)$ has been constructed. 
		We now move to the construction of the triple at the stage~\(j+1\) which, as before, will depend on the value of $\epsilon_{k-j}$. 
		Suppose $\epsilon_{k-j}=1$. The case $\epsilon_{k-j}=-1$ can be handled analogously and is therefore omitted.
		Similarly to the definition of~\(V_1\), define
		\[
		\begin{split}
			V_{j+1}\defeq V_j\cap &\{(g_n)_n\colon f((g_n)_n)c_j=f((g^{(j)}_n)_n)c_j\text{ for all $f\in F$}\} \\
			&\cap F_{w_{1,k+1-j}x_1,w_{k+2-j,k+1}((g^{(j)}_n)_n)c_j}.
		\end{split}
		\]
		The nonemptiness of~\(V_{j+1}\) follows from the induction hypothesis applied to the word $w_{1,k+1-j}x_1$. Consider~$(g_n^{(j+1)})_n\in V_{j+1}$. Then
		\[
		(w_{1,k+1-j}((g_n^{(j+1)})_n)g^{(j+1)}_1)|_{w_{k+2-j,k+1}((g_n^{(j)})_n)c_j}\neq \id|_{w_{k+2-j,k+1}((g_n^{(j)})_n)c_j}, 
		\]
		which by construction of $V_{j+1}$ is equivalent to 
		\[
		(w_{1,k+1-j}((g_n^{(j+1)})_n)g^{(j+1)}_1)|_{w_{k+2-j,k+1}((g_n^{(j+1)})_n)c_j}\neq \id|_{w_{k+2-j,k+1}((g_n^{(j+1)})_n)c_j}. 
		\]
		Arguing as before, we obtain $0<c_{j+1}\leq c_j$ such that $e_{j+1}^1((g_n^{(j+1)})_n;c_{j+1})=0$. It is easy to check that the conditions \ref{cond.i}--\ref{cond.v} are satisfied by the triple  $(V_{j+1},c_{j+1},(g^{(j+1)}_n)_n)$.
		
		The conditions \ref{cond.i}--\ref{cond.v} also give that $e^{\epsilon_{k+1-j}}_j((g^{(k)}_n)_n;c_k)=0$ for~\(j=1,\dots,k\). Indeed, this follows by construction when $j=k$, while for $j<k$, we have
		\[
		e^{\epsilon_{k+1-j}}_j((g^{(k)}_n)_n;c_k)\leq  e^{\epsilon_{k+1-j}}_j((g^{(k)}_n)_n;c_j)= e^{\epsilon_{k+1-j}}_j((g^{(j)}_n)_n;c_j)=0,
		\]
		where in the inequality we used the fact that $c_k\leq c_j$ and in the equality we used that $(g^{(k)}_n)_n\in V_{j+1}$ and that all the words appearing in $e_j^{\epsilon_{k+1-j}}$ are in $F$.
		
		We are now ready to prove that $V\cap F_{w,c}\neq\emptyset$. In fact, we shall show the stronger condition $V_k\cap F_{w,c_k}\neq\emptyset$. To simplify the notation, we denote by $(g_n)_n$ the sequence $(g^{(k)}_n)_n\in V_k$. Moreover, let $V'=V_k$ and $c'=c_k$. If $w((g_n)_n)|_{c'}\neq \id|_{c'}$, then $(g_n)_n\in V\cap F_{w,c'}$ and we are done, so assume $w((g_n)_n)|_{c'}= \id|_{c'}$. Then the hypotheses of Lemma~\ref{lem:simplifytopfreeness} are satisfied for the word $w_{1,k+1}$. Therefore, there exist $g'\in G$ and nonzero disjoint $d_1,d_2\leq hc'$ such that
		
		\begin{enumerate}
			\item $(g',g_2,\dots)\in V'$;
			\item\label{cond.2} $g'|_{B\downarrow 1\setminus hc'}=g_1|_{B\downarrow 1\setminus hc'}$ and $(g')\inv|_{B\downarrow 1\setminus c'}=g_1\inv|_{B\downarrow 1\setminus c'}$;
			\item $g'(d_1)=g_1(d_2)$ and $g'(d_2)=g_1(d_1)$;
		\end{enumerate}
		where $h=w_{1,k+1}((g_n)_n)$.
		
		We claim that~$w_{1,k+1}(g_1,g_2,\dots)|_{c'}=w_{1,k+1}(g',g_2,\dots)|_{c'}$. To this end, we show by induction on $j=0,\dots,k$ that $w_{k+1-j,k+1}(g_1,g_2,\dots)|_{c'}=w_{k+1-j,k+1}(g',g_2,\dots)|_{c'}$. The case $j=0$ follows from the fact that $x_1$ does not occur in $w_{k+1}$. 
		We will now prove that the claim holds for~\(j+1\) supposing that it holds for~\(j\). We have
		\[
		w_{k-j,k+1}((g_n)_n)|_{B\downarrow c'}=w_{k-j}((g_n)_n)g_1^{\epsilon_{k-j}}w_{k+1-j,k+1}((g_n)_n)|_{B\downarrow c'}.
		\]
		Since $e_{j+1}^{\epsilon_{k-j}}((g_n)_n;c')=0$, it follows that either $w_{k+1-j,k+1}((g_n)_n)c'\leq 1\setminus hc'$ or~$w_{k+1-j,k+1}((g_n)_n)c'\leq 1\setminus c'$ depending on the value of~\(\epsilon_{k-j}\). 
		Combined with the inductive hypothesis and condition (\ref{cond.2}), this gives 
		\[
		g_1^{\epsilon_{k-j}}w_{k+1-j,k+1}((g_n)_n)|_{B\downarrow c'}=(g')^{\epsilon_{k-j}}w_{k+1-j,k+1}((g',g_2,\dots))|_{B\downarrow c'}.
		\]
		Since $x_1$ does not appear in $w_{k-j}$, we finally have
		\[
		w_{k-j,k+1}(g_1,g_2,\dots)|_{B\downarrow c'}=w_{k-j,k+1}(g',g_2,\dots)|_{B\downarrow c'},
		\]
		as desired.
		
		We now prove that $w(g',g_2,\dots)|_{c'}\neq \id|_{c'}$, showing that $(g',g_2,\dots)\in V'\cap F_{w,c'}$. Indeed, for the element $g_1(d_1)\leq c'$, we have
		\[
		\begin{split}
			w(g',g_2,\dots)(g_1(d_1))&=g'w_{1,k+1}(g',g_2,\dots)(g_1(d_1))\\
			&=g'w_{1,k+1}(g_1,g_2,\dots)(g_1(d_1))\\
			&=g'h(g_1(d_1))=g'(d_1)=g_1(d_2)\neq g_1(d_1).
		\end{split}
		\]
		This shows that  $w\in W$, concluding the proof.
	\end{proof}
	
	Combining Proposition \ref{thm.generictopfreenessnew} together with Baire category theorem we obtain:
	
	\begin{cor}\label{thm.generictopfreeness}
		The set 
		\begin{align*}
			F=\bigcap_{w\in \mathbb F_\infty\setminus\{e\}}\bigcap_{c\in B\setminus\{0\}}F_{w,c}
		\end{align*}
		is a dense~$G_\delta$ subset of~\(\Aut(B,\mu)^\N\). 
	\end{cor}
	
	\begin{thm}\label{thm.genericnocomp}
		Let \(M\) be a countable $2$-divisible simple refinement cone which is not almost unperforated, and~$u\in M$ be a nonzero element. Let~$(B,\mu)$ be the \Fraisse limit of finite~$(M,u)$-measured Boolean algebras.
		A generic countable subgroup of~\(\Aut(B,\mu)\) is dense, isomorphic to~\(\Finf\), and acts topologically freely on~\(\St(B)\).
		
		In particular, there exist topologically free minimal actions of~$\mathbb F_\infty$ on the Cantor space without dynamical comparison.
	\end{thm}
	\begin{proof}
		A generic element~$(g_n)_n\in\Aut(B,\mu)^\N$ generates a dense subgroup by Lemma \ref{lem.dense} and it is in the subset~$F$ defined in Corollary \ref{thm.generictopfreeness}. Let~$H=\langle g_n\colon n\in \N\rangle$. Then by Theorem \ref{thm.minactnocomp},~$H\acts \St(B)$ is minimal and does not have dynamical comparison. Consider the surjective group homomorphism~$\phi\colon \mathbb F_\infty\to H$ extending the map~$x_i\mapsto g_i$. Then for every nontrivial word~$w\in\mathbb F_\infty$, the condition~$(g_n)_n\in F$ implies that the element~$w((g_n)_n)$ acts topologically freely, that is, it does not fix pointwise any clopen subset. In particular, it is not the identity. This shows that~$\phi$ is injective, hence~$H\cong \mathbb F_\infty$, and that the action $H\acts \St(B)$ is topologically free. 
	\end{proof}
	
	In Theorem \ref{thm.constructionP} we showed the existence of a cone~$P$ satisfying the properties of Theorem \ref{thm.genericnocomp} and admitting a V-homomorphism~$\tau\colon P\to (\mathbb Q_2)_+$. Let $u$ be a nonzero element of $P$, then, after a normalisation, we can assume that $\tau(u)=1$. Let $(B,\mu)$ be the \Fraisse limit of finite $(P,u)$-measured Boolean algebras. Since~$P$ is not isomorphic to~$\N$, Proposition \ref{prop.finiteoratomless} implies that~$\St(B)$ is homeomorphic to the Cantor space~$X$. The V-homomorphism~$\tau$ induces a measure~$\nu=\tau\circ \mu\colon B\to (\mathbb{Q}_2)_+$. Suppose~$\nu(a)=x_1+x_2$, then by the Vaught property of~$\tau$, we have that~$\mu(a)=y_1+y_2$ for some~$y_1,y_2\in P$ satisfying $\tau(y_i)=x_i$. By the Vaught property of $\mu$, there is a partition $a=a_1\oplus a_2$ with $\mu(a_i)=y_i$, hence $\nu(a_i)=x_i$. This shows that~$\nu$ is a V-measure on~$(\mathbb{Q}_2)_+$, and moreover $\nu(1)=1$. Since the~$(\frac{1}{2},\frac{1}{2})$-Bernoulli measure~$\lambda$ is also a V-measure normalised at $1$, by uniqueness of the V-measure \cite[Theorem 4.6.8]{Wehrung:Monoids_Boolean}, there is an automorphism~$\Phi$ of~$B$ such that~$\lambda = \nu\circ\Phi$. Since~$\lambda = \nu \circ \Phi = \tau \circ \mu \circ \Phi$, replacing~$\mu$ 
	by~$\mu \circ \Phi$ and conjugating the action of~$\Aut(B,\mu)$ by~$\Phi$, we obtain an isomorphic copy of~$\Aut(B,\mu)$ inside~$\Aut(B,\lambda)$. 
	Without loss of generality we may therefore assume that every element of $\Aut(B,\mu)$ preserves~$\lambda$, so that every countable subgroup of $\Aut(B,\mu)$ acts on~$\St(B)$ preserving the Bernoulli measure~$\lambda$. Combining this observation with Theorem \ref{thm.genericnocomp} we get:
	
	\begin{cor}\label{cor:top_free_measure_preserving}
		There exist topologically free minimal actions of~$\mathbb F_\infty$ on the Cantor space that preserve the Bernoulli~$(\frac{1}{2},\frac{1}{2})$-measure and fail dynamical comparison.
	\end{cor}
	
	\begin{rmk}\label{rmk:not_Z_stable}
		By the same argument as the one in \cite[Proposition 4.7]{bell:buncededdens}, it can be shown that the crossed product of any minimal topologically free probability-measure-preserving action of a group which is not inner amenable is not $\mathcal Z$-stable. In particular, since $\mathbb F_\infty$ is not inner amenable, the crossed product of the action built in Corollary \ref{cor:top_free_measure_preserving} is not $\mathcal Z$-stable.
	\end{rmk}
	
	Similar properties can also occur in the absence of invariant measures. In Proposition \ref{thm.constructionQ}, we showed the existence of countable divisible simple refinement cones that are not almost unperforated and have no nontrivial states. By Theorem \ref{thm.genericnocomp}, a generic countable subgroup of~$\Aut(B,\mu)$ is isomorphic to~$\mathbb F_\infty$, and its action on~$\St(B)$ is topologically free, minimal and fails dynamical comparison. By the proof of Theorem \ref{thm.minactnocomp}, the action of a generic countable subgroup of~$\Aut(B,\mu)$ has the same invariant measures of the action of~$\Aut(B,\mu)$ itself. Since the type semigroup of~$\Aut(B,\mu)\acts \St(B)$ admits no nontrivial state, the action has no invariant probability measures by Remark \ref{rmk:states=measures}, yielding the following result.
	
	\begin{cor}\label{cor:action_no_measure}
		There exist topologically free minimal actions of~$\mathbb F_\infty$ on the Cantor space without invariant measures that fail dynamical comparison.
	\end{cor}
	
	\begin{rmk}\label{rmk:plain_paradoxes}
		The cone~\(Q\) of Proposition~\ref{thm.constructionQ} used to construct the minimal Cantor actions in Corollary~\ref{cor:action_no_measure} not only fail almost unperfroation but also fail plain paraxodes: \((n+1)x\le n x\) implies~\(2x\le x\) for every~\(n\geq 1\), \(x\in Q\) (see \cite[Definition 2.18]{Kwasniewski-Meyer-Prasad:Type_semigroups}). 
		The type semigroups of these actions are precisely~\(Q\) by Lemma~\ref{lem:type_dense_subgroup} and hence have no plain paradoxes.
		However, the type semigroups associated to the minimal actions constructed in Corollary~\ref{cor:top_free_measure_preserving} do have plain paradoxes in a trivial way.
	\end{rmk}
	
	\begin{rmk}
		The (topologically free) minimal actions of~\(\Finf\) constructed in this section are necessarily not amenable. Indeed,  \cite[Theorem B]{GGKN_cp-nonamenable} implies that all minimal amenable actions of non-amenable free groups on the Cantor space have dynamical comparison. Consequently, genericity arguments cannot be used to construct an amenable counterexample. 
	\end{rmk}

	
	\subsection{Additional dynamical properties}
	
	We now shift our attention to other dynamical properties not detected by the type semigroup: topological weak mixing and conjugacy to a subshift. The results in this subsection crucially rely on the cancellativity property.
	
	\begin{defn}
		An action $G\acts X$ is \emph{topologically weakly mixing} if for any nonempty open subsets~$U_1,U_2,V_1,V_2\subseteq X$ there is an element~$g\in G$ such that both~$gU_1\cap V_1$ and~$gU_2\cap V_2$ are nonempty.
	\end{defn}
	
	\begin{lem}\label{lem.formixing}
		Let~$B$ be a Boolean algebra, let~$M$ be a countable~$2$-divisible simple refinement cone, and let~$\mu$ be an~$M$-valued V-measure. Let~$a_1, a_2, b_1, b_2$ be nonzero elements of~$B$.
		Then there exist nonzero elements~$a_i' \leq a_i$ and~$b_i' \leq b_i$ such that
		$\mu(a_i') = \mu(b_i')$ for~$i = 1, 2$, and~$a_1' \wedge a_2' = 0 = b_1' \wedge b_2'$.
	\end{lem}
	\begin{proof}
		Without loss of generality, we can assume that~$\mu(a_i) = \mu(b_i)$ for~$i = 1, 2$. Indeed, since~$M$ is a simple refinement cone, there is an element~$m\in M$ that is below both~$\mu(a_i)$ and~$\mu(b_i)$, and then use the Vaught property to replace~$a_i$ and~$b_i$ by smaller elements of measure~$m$. 
		We consider several cases.\\
		\medskip
		\noindent\textit{Case 1:~$b_1 \wedge b_2 = 0$ and~$a_1 \neq a_2$.}
		Without loss of generality,~$a_1 \setminus a_2 \neq 0$; set~$a_1' = a_1 \setminus a_2$. Then
		\[
		\mu(b_1) = \mu(a_1) = \mu(a_1') + \mu(a_1 \wedge a_2).
		\]
		Since~$\mu$ is a V-measure, we may decompose~$b_1 = b_1' \vee b_1''$ with~$b_1', b_1''$ disjoint and~$\mu(b_1') = \mu(a_1')$. Set~$a_2' = a_2$ and~$b_2' = b_2$. Then $a_1' \wedge a_2' = 0$ by construction, and~$b_1' \wedge b_2' = 0$ since~$b_i' \leq b_i$ and~$b_1 \wedge b_2 = 0$.
		
		\medskip
		\noindent\textit{Case 2:~$b_1 \wedge b_2 = 0$ and~$a_1 = a_2$.}
		In this case~$\mu(a_1) = \mu(b_1) = \mu(b_2)$. Since~$M$ is~$2$-divisible, there exist disjoint $a_1', a_2' \leq a_1$ with~$a_1 = a_1' \vee a_2'$ and~$\mu(a_1') = \mu(a_2') = \mu(a_1)/2$.
		Applying the V-measure property to~$b_1$ and~$b_2$ separately, we obtain decompositions $b_i = b_i' \vee b_i''$ with~$\mu(b_i') = \mu(a_i')$ for~$i = 1, 2$. Since~$b_1 \wedge b_2 = 0$, we have~$b_1' \wedge b_2' = 0$, and by construction~$a_1' \wedge a_2' = 0$.
		
		\medskip
		\noindent\textit{Case 3:~$a_1 \wedge a_2 = 0$.}
		This is symmetric to Cases 1 and 2, with the roles of~$a_i$ and~$b_i$ interchanged.
		
		\medskip
		\noindent\textit{Case 4:~$a_1 \wedge a_2 \neq 0$ and~$b_1 \wedge b_2 \neq 0$.}
		Suppose first that~$a_1\neq a_2$. After interchanging~$a_1$ and~$a_2$ if necessary, we may assume that \[ c\defeq a_1\setminus a_2\neq 0. \] Since~$\mu(c)\leq\mu(a_1)=\mu(b_1)$, the Vaught property provides a nonzero element~$d\leq b_1$ such that~$\mu(d)=\mu(c)$. We may therefore apply Case~3 to the quadruple~$c,a_2,d,b_2$, and this yields the desired elements. Similarly, suppose that~$b_1\neq b_2$. After interchanging~$b_1$ and~$b_2$ if necessary, set \[ d\defeq b_1\setminus b_2\neq 0. \] Since~$\mu(d)\leq\mu(b_1)=\mu(a_1)$, the Vaught property provides a nonzero element~$c\leq a_1$ such that~$\mu(c)=\mu(d)$. We may then apply Case~1 or Case~2 to the quadruple~$c,a_2,d,b_2$, according to whether~$c\neq a_2$ or~$c=a_2$. It remains to consider the case~$a_1=a_2$ and~$b_1=b_2$. Since~$M$ is~$2$-divisible, there is an element~$m\in M$ such that~$\mu(a_1)=2m$. As~$\mu(a_1)=\mu(b_1)$, the Vaught property gives decompositions \[ a_1=a_1'\oplus a_2' \quad\text{and}\quad b_1=b_1'\oplus b_2' \] such that \[ \mu(a_1')=\mu(a_2')=\mu(b_1')=\mu(b_2')=m. \] These elements satisfy the required conclusions.
	\end{proof}
	
	\begin{prop}\label{prop.mixing}
		Let~$M$ be a cancellative~$2$-divisible countable simple refinement cone, and let~$u\in M\setminus{0}$. Let ~$(B,\mu)$ be the \Fraisse limit of finite~$(M,u)$-measured Boolean algebras. The action of a generic countable subgroup of~$\Aut(B,\mu)$ on~$\St(B)$ is topologically weakly mixing.
	\end{prop}
	
	\begin{proof}
		Let
		\begin{align*}
			W\defeq \{(g_n)_n\in \Aut(B,\mu)^\N &\colon \forall a_1,a_2,b_1,b_2\in B\setminus\{0\}~\exists n\in\N \\
			&\text{such that } g_na_1\cap b_1\neq\emptyset,g_na_2\cap b_2\neq\emptyset\}.
		\end{align*}
		We show that~$W$ is a dense~$G_\delta$ in~$\Aut(B,\mu)^\N$. We can write
		\[
		W=\bigcap_{a_1,a_2,b_1,b_2\in B\setminus \{0\}}\bigcup_{n\in\N}W_{a_1,a_2,b_1,b_2,n},
		\]
		where
		\(
		W_{a_1,a_2,b_1,b_2,n}\defeq \{(g_n)_n\colon g_na_1\cap b_1\neq\emptyset,g_na_2\cap b_2\neq\emptyset\}.
		\)
		Since~$W_{a_1,a_2,b_1,b_2,n}$ is open, it suffices to show that each~$W_{a_1,a_2,b_1,b_2}=\bigcup_{n\in\N}W_{a_1,a_2,b_1,b_2,n}$ is dense. Since basic open sets impose restrictions only on finitely many coordinates, it suffices to show that for each~$a_1,a_2,b_1,b_2\in B$ there is~$g\in\Aut(B,\mu)$ satisfying~$ga_1\cap b_1\neq\emptyset$ and~$ga_2\cap b_2\neq\emptyset$. By Lemma~\ref{lem.formixing}, there exist nonzero~$a_i'\leq a_i$ and 
		$b_i'\leq b_i$ such that~$a_1'\wedge a_2'=0=b_1'\wedge b_2'$ and $\mu(a_1')=\mu(b_1')$,~$\mu(a_2')=\mu(b_2')$. By cancellativity of~$M$ and Theorem \ref{thm.Fraisseprop}, there exists~$g\in\Aut(B,\mu)$ with~$g(a_1')=b_1'$ and~$g(a_2')=b_2'$. 
		Then~$g(a_1)\cap b_1\geq b_1'$, so~$g(a_1)\cap b_1\neq 0$, and similarly~$g(a_2)\cap b_2\neq 0$.
	\end{proof}
	
	Combining Proposition \ref{prop.mixing}, Theorem \ref{thm.minactnocomp}, Theorem \ref{thm.genericnocomp}, and the discussion following the latter, we obtain
	
	\begin{cor}
		There exist topologically free topologically weakly mixing minimal actions of~$\mathbb F_\infty$ on the Cantor space that preserve the Bernoulli~$(\frac{1}{2},\frac{1}{2})$-measure and fail dynamical comparison.
	\end{cor}

	We will now show that when $M$ is a countable $2$-divisible simple refinement cone which is not almost unperforated and is additionally cancellative, the actions constructed in Theorem \ref{thm.genericnocomp} are conjugate to a subshift of $\{0,1\}^{H}$.

	Since $M$ is 2-divisible, we can fix an element $v\in M$ such that $u=2v$. The key observation is that the orbit of an element of measure $v$ under the action of a dense subgroup of $\Aut(B,\mu)$ generates the whole algebra $B$. We start by showing that $B$ is generated by all the elements of measure $v$.
	\begin{lem}\label{lem:measuregeneration}
		The Boolean algebra generated by $\{d\in B\colon \mu(d)=v\}$ is $B$. 
	\end{lem}
	\begin{proof}
		Let $b\in B$. Since $v$ is an order unit there are $s\in M$ and $k\in \N$ such that $\mu(b)+s=kv$. By the refinement property there are elements $m_1,\dots,m_k\leq v$ such that $\mu(b)=m_1+\dots+m_k$. Applying the Vaught property again gives a partition $b=b_1\oplus\dots\oplus b_k$ with $\mu(b_i)=m_i$. Fix $i=1,\dots, k$. Then there is $s_i$ such that $m_i+s_i=v$. Therefore
		\[
		m_i+\mu(\neg b_i)=\mu(b_i)+\mu(\neg b_i)=u=2v=2m_i+2s_i.
		\]
		Cancellativity gives $\mu(\neg b_i)=m_i+2s_i$. By the Vaught property we obtain a partition $\neg b_i=c+e_1+e_2$ with $\mu(c)=m_i$ and $\mu(e_j)=s_i$. Set $d_{i,1}=b_i\oplus e_{i,1}$ and $d_{i,2}=c\oplus e_{i,1}$. Then $\mu(d_{i,1})=\mu(d_{i,2})=v$ and $b_i=d_{i,1}\setminus d_{i,2}$. Since $b=b_1\oplus\dots\oplus b_k$, it belongs to the Boolean algebra generated by elements of measure $v$.
	\end{proof}
	
	\begin{lem}\label{lem:orbitgenerates}
		Let $a\in B$ with $\mu(a)=v$, and let $H\leq \Aut(B,\mu)$ be a dense subgroup. The Boolean algebra generated by $\{ga\colon g\in H\}$ is $B$. 
	\end{lem}
	\begin{proof}
		By Lemma \ref{lem:measuregeneration} it suffices to show that the Boolean algebra generated by the $H$-orbit of $a$ contains all the elements of measure $v$. Let $b\in B$ with $\mu(b)=v$. Since $P$ is cancellative, $\mu(\neg a)=\mu(\neg b)=v$. Theorem \ref{thm.Fraisseprop} implies that there is $g\in \Aut(B,\mu)$ such that $ga=b$. Density of $H$ allows to choose $h\in H$ with $ha=b$.
	\end{proof}
	
	\begin{prop}\label{thm:subshift}
		Let $H\leq \Aut(B,\mu)$ be a dense subgroup. Then $H\acts \St(B)$ is conjugate to a subshift of $\{0,1\}^{H}$.
	\end{prop}
	\begin{proof}
		Let $a\in B$ with $\mu(a)=v$ and consider the map 
		\[
		\Phi\colon X\to \{0,1\}^{H}
		\]
		defined by $\Phi(x)(g)=1_{ga}(x)$.  The map $\Phi$ is continuous and 
		\[
		\Phi(hx)(g)=1_{ga}(hx)=1_{h\inv ga}(x)=\Phi(x)(h\inv g)=h\Phi(x),
		\]
		hence $\Phi$ is $H$-invariant. By Lemma \ref{lem:orbitgenerates}, the map $\Phi$ is injective. Hence $H\acts \St(B)$ is conjugate to a subshift of $\{0,1\}^{H}$.
	\end{proof}

	
	\section{Crossed products}\label{sec.cp} 
	
	The measure-preserving topologically free actions of~\(\Finf\) built in Section~\ref{sec.minnocomp} are not \(\Z\)-stable, as noted in Remark~\ref{rmk:not_Z_stable}. 	
	However, the actions of Section \ref{sec.minnocomp} can be modified to construct simple \(\mathcal{Z}\)-stable crossed products arising from actions that lack dynamical comparison.
	Given an action of~\(\Finf\) on the Cantor space~\(X\), one can take the tensor product of the resulting crossed product with a \(\mathcal{Z}\)-stable C*-algebra~\(A\). The resulting C*-algebra~\(A\otimes (\Finf\ltimes_\red C(X))\) will be \(\Z\)-stable. The C*-algebra~\(A\) can be chosen to be the reduced group C*-algebra of~\(\NF\) which is simple and \(\Z\)-stable. The resulting tensor product will also be a crossed product given by the action of~\(\NF\) on~\(X\) which factors through the action~\(\Finf\acts X\). This results in a simple \(\Z\)-stable crossed product resulting from an action which lacks dynamical comparison. Since \(\Z\)-stability implies strict comparison of positive elements for simple unital C*-algebras, this gives the first example of a crossed product which has strict comparison but the underlying action lacks dynamical comparison. Moreover, this can happen more severely: both in the presence and absence of traces.
	
	We start by showing the folklore result that~$\Cst_\red(\NF)$ is simple and~$\mathcal{Z}$-stable. We include a proof as we could not find it in the literature.
	
	\begin{lem}\label{prop.directsumiscssimple}
		Let~$(G_n)_{n=1}^\infty$ be a nested sequence of countable discrete groups, and let~$G=\bigcup G_n$ be its direct union. Then~$\cs_\red(G)$ is the direct limit of the algebras~$\cs_\red(G_n)$. In particular, a direct union of~$\cs$-simple groups is~$\cs$-simple.
	\end{lem}
	\begin{proof}
		The group algebra~$\mathbb{C}[G]$ is the algebraic direct limit of the algebras~$\mathbb C[G_n]$. Therefore it suffices to show that for all~$n$ the norm~$\left\|\cdot\right\|_{G_n}$ induced on~$\mathbb C[G_n]$ by the left regular representation on~$\ell^2(G_n)$ is equal to~$\left\|\cdot\right\|_G$, the one induced by the left regular representation on~$\ell^2(G)$. Since~$G=\bigsqcup_{[k]\in G/G_n} G_nk$, there is a decomposition~$\ell^2(G)=\bigoplus_{[k]\in G/G_n}\ell^2(G_nk)$. Given an element~$x\in \mathbb C[G_n]$, the action of~$x$ on~$\ell^2(G)$ decomposes as a direct sum of actions on~$\ell^2(G_nk)$, each of norm~$\left\|x\right\|_{G_n}$. Therefore, ~$\left\|x\right\|_{G}=\|x\|_{G_n}$ as desired.
	\end{proof}
	
	\begin{lem}\label{lem.embintensors}
		Let~$A_0$ be a separable unital~$\cs$-algebra,~$A=\bigotimes_{n\in\N}A_0$ and~$A_\infty=\prod_\N A/\bigoplus_\N A$ be the sequence algebra of~$A$. 
		Then~$A_0$ embeds unitally in~$A_\infty\cap A'$, where~$A'$ denotes the commutant of~$A$ in~$A_\infty$.
	\end{lem}
	
	\begin{proof}
		For each~$k \geq 1$, let~$\iota_k \colon A_0 \to A$ denote the canonical
		unital~$*$-homomorphism embedding~$A_0$ into the~$k$-th tensor factor:
		\[
		\iota_k(a) = 1 \otimes \cdots \otimes 1 \otimes a \otimes 1 \otimes \cdots,
		\]
		with~$a$ in position~$k$. Since each~$\iota_k$ is an isometric
		$*$-homomorphism, the sequence~$(\iota_k(a))_{k \geq 1}$ induces a unital embedding of~$A_0$ in~$A_\infty$ given by~$\Phi(a)=[\iota_k(a)]_k$.
		Then~$\Phi(A_0)\subseteq A'$, that is,  
		\[
		\lim_{k \to \infty} \|[\iota_k(a), b]\| = 0
		\quad \text{for all } a \in A_0,\ b \in A.
		\]
		Indeed, this follows immediately when~$b$ is in the algebraic tensor product, and the general case follows from a standard density argument.
	\end{proof}
	
	\begin{prop}
		The reduced~$\cs$-algebra~$\cs_\red (\bigoplus_\N \mathbb F_\infty)$ is simple and~$\mathcal Z$-stable. 
	\end{prop}
	
	\begin{proof}
		Simplicity easily follows from Lemma \ref{prop.directsumiscssimple}, and~$\cs$-simplicity of~$\mathbb F_\infty$, so we focus on~$\mathcal{Z}$-stability. By a result of Toms--Winter \cite[Theorem 2.2]{Toms-Winter:Strongly_self-abs_Cstar_algs}, a unital separable~$\cs$-algebra~$A$ is~$\mathcal Z$-stable if and only if the Jiang--Su algebra~$\mathcal Z$ embeds unitally in the central sequence algebra~$\left(\prod_\N A/\bigoplus_\N A\right)\cap A'$. Since~$\mathcal Z$ unitally embeds in~$\cs_\red(\mathbb F_\infty)$ by \cite[Proposition 4.2]{Thiel-Winter:Generator_problem}, it suffices to show that~$\cs_\red(\mathbb F_\infty)$ unitally embeds in the central sequence algebra of~$\cs_\red(\bigoplus_\N\mathbb F_\infty)$. This follows from Lemma \ref{lem.embintensors} by noticing that~$\cs_\red(\bigoplus_\N \mathbb F_\infty)=\bigotimes_\N \cs_\red(\mathbb F_\infty)$.
	\end{proof}
	
	The previous proposition shows that~$\bigoplus_\N \mathbb F_\infty$ is~$\cs$-simple. By \cite[Theorem 1.8]{BKKO_C*simplicity}, minimal actions of~\(\NF\) give rise to simple reduced crossed products. Moreover,~$\bigoplus_\N\mathbb{F}_\infty$ is an exact group. By using the Kirchberg--Wasserman definition of exact groups \cite{KW_ExactGroups}, it is easy to see that crossed products arising from actions of exact groups are exact as~$\cs$-algebras.
	
	\begin{thm}
		There exist minimal actions of~$\bigoplus_\N\mathbb F_\infty$ on the Cantor space~\(X\) that fail dynamical comparison and such that either: 
		\begin{enumerate}
			\item~$\bigoplus_\N\mathbb F_\infty\ltimes_{\red}C(X)$ is simple, exact,~$\mathcal Z$-stable and with a trace, hence it is stably finite; or
			\item~$\bigoplus_\N\mathbb F_\infty\ltimes_{\red }C(X)$ is simple, exact,~$\mathcal Z$-stable and traceless, hence it is purely infinite.
		\end{enumerate}
	\end{thm}
	\begin{proof}
		Let~$\alpha',\beta' \colon \mathbb F_\infty\acts X$ be minimal actions lacking dynamical comparison such that $\alpha'$ has an invariant measure, while $\beta'$ has none (the existence of such actions was shown in Section \ref{sec.minnocomp}).
		Consider the actions~$\alpha,\beta\colon \bigoplus_\N\mathbb F_\infty\acts X$ given by~$\alpha(g_0, g_1,\dots)x=\alpha'(g_0)x$, and~$\beta(g_0, g_1,\dots)x=\beta'(g_0)x$. These actions are minimal and lack dynamical comparison. 
		Moreover,~$\alpha$ has an invariant measure, while~$\beta$ has none.
		
		The reduced crossed products resulting from the actions are
		\begin{align*}
			A\defeq\NF\ltimes_{\alpha,r} C(X)&=(\Finf\ltimes_{\alpha',r} C(X))\otimes_{\min}\cs\big(\NF\big)\quad\text{and}\\
			B\defeq\NF\ltimes_{\beta,r} C(X)&=(\Finf\ltimes_{\beta',r} C(X))\otimes_{\min}\cs\big(\NF\big).
		\end{align*}
		
		These crossed products are simple by minimality of the actions and C*-simplicity of the acting group (see \cite[Theorem 1.8]{BKKO_C*simplicity}).
		Further, they are \(\Z\)-stable since so is~$\cs(\bigoplus_\N\mathbb F_\infty)$. 
		The C*-algebra~\(A\) has a trace owing to~\(\alpha\) admitting an invariant measure.
		Similarly,~\(B\) admits no trace.
		Stable finiteness and pure infiniteness follow from exactness of the crossed products and Kirchberg's dichotomy (see \cite[Theorem 4.4.10]{Ror_Classification}).
	\end{proof}
	
	Unital simple~$\mathcal{Z}$-stable~$\cs$-algebras are known to have almost unperforated Cuntz semigroup by a result of R\o{rdam} \cite[Theorem 4.5]{Rordam:Z-stablestrictcomp}. Since in this setting almost unperforation of the Cuntz semigroup is equivalent to strict comparison by \cite[Remark 9.2(3)]{Thiel:Ranks-of-operators}, we get:
	
	\begin{cor}
		There exist minimal actions of~$\bigoplus_\N\mathbb F_\infty$ on the Cantor space that fail dynamical comparison and such that the crossed product is simple and has strict comparison. This phenomenon can happen both when the crossed product is stably finite and when it is purely infinite. In particular, strict comparison of the simple crossed product does not imply dynamical comparison of the action.
	\end{cor}
	
	This corollary complements the result obtained by Naryshkin in \cite[Remark 3.4(ii)]{Naryshkin_Pol-growth-comparison-SBP}, where it is shown that dynamical comparison of the action does not imply strict comparison of the crossed product.

	
	\bibliography{thebibliography}
	
\end{document}